\documentclass[12pt]{amsart}

\newtheorem{theorem}{Theorem}[section]
\newtheorem{proposition}[theorem]{Proposition}
\newtheorem{lemma}[theorem]{Lemma}
\newtheorem{corollary}[theorem]{Corollary}

\theoremstyle{definition}
\newtheorem{definition}[theorem]{Definition}

\newtheorem{example}[theorem]{Example}
\newtheorem{problem}[theorem]{Problem}

\newtheorem{remark}[theorem]{Remark}

\newcommand{\ZZ}{ \ensuremath{\mathbb{Z}}}
\newcommand{\CC}{ \ensuremath{\mathbb{C}}}
\newcommand{\RR}{ \ensuremath{\mathbb{R}}}

\newcommand{\rlex}{{\mathrm{{rev}}}}

\newcommand{\rank}{\ensuremath{\mathrm{rank}}\hspace{1pt}}

\def\cocoa{{\hbox{\rm C\kern-.13em o\kern-.07em C\kern-.13em o\kern-.15em A}}}

\newcommand{\hh}{\mathcal{H}}
\newcommand{\pd}{\mathsf{pd}}
\newcommand{\del}{\mathrm{del}}
\newcommand{\D}{\mathcal{D}}
\newcommand{\color}{\mathrm{color}}

\begin{document}

\title[Face vectors of simplicial cell manifolds]
{Face vectors of simplicial cell decompositions\\ of manifolds}

\author{Satoshi Murai}
\address{
Satoshi Murai,
Department of Mathematical Science,
Faculty of Science,
Yamaguchi University,
1677-1 Yoshida, Yamaguchi 753-8512, Japan.
}
%\email{murai@yamaguchi-u.ac.jp}

\thanks{This work was supported by KAKENHI 22740018}

%\author{Another Author}
%\address{
%}
%\email{}

%Keyword and Subject Classes (if needed)
%\keywords{}
%\subjclass[2000]{}

\begin{abstract}
In this paper, we study face vectors of simplicial posets that are the face posets of cell decompositions of  topological manifolds without boundary.
We characterize all possible face vectors of simplicial posets whose geometric realizations are homeomorphic to the product of spheres.
As a corollary,
we obtain the characterization of face vectors of simplicial posets whose geometric realizations are
odd dimensional manifolds without boundary.
\end{abstract}

\maketitle

\section{Introduction}
The study of face numbers is one of the central topics in combinatorics.
A goal of the study is to obtain characterizations of classes of face vectors of certain combinatorial objects.
In this paper,
we study face vectors of simplicial posets,
particularly those whose geometric realizations are manifolds.

A \textit{simplicial poset} is a finite poset $P$ with a minimal element $\hat 0$
such that every interval $[\hat 0, \sigma]$ for $\sigma \in P$ is a boolean algebra.
It is known that any simplicial poset is the face poset of a regular CW-complex $\Gamma(P)$ \cite{Bj}.
A CW-complex whose face poset is a simplicial poset
is called a \textit{simplicial cell complex} (also called a boolean cell complex
or a pseudocomplex).

Let $P$ be a simplicial poset.
We say that an element $\sigma \in P$ has \textit{rank $i$}, denoted $\rank \sigma =i$,
if $[\hat 0,\sigma]$ is a boolean algebra of rank $i+1$.
Thus those elements correspond to $(i-1)$-dimensional cells of $\Gamma(P)$.
The \textit{dimension} of $P$ is 
$$\dim P= \max \{\rank \sigma : \sigma \in P\}-1.$$
Let $f_i=f_i(P)$ be the number of elements $\sigma \in P$ having rank $i$
and $d=\dim P +1$.
The vector $f(P)=(f_0,f_1,\dots,f_d)$ is called the \textit{$f$-vector of $P$}.
To study $f$-vectors,
it is often convenient to consider the \textit{$h$-vector $h(P)=(h_0,h_1,\dots,h_d)$
of $P$} defined by
$$\sum_{i=0}^d f_i t^i (1-t)^{d-i} = \sum_{i=0}^d h_i t^i.$$
It is easy to see that knowing $f(P)$ is equivalent to knowing $h(P)$.

A \textit{simplicial cell sphere} is a simplicial poset $P$ such that $\Gamma(P)$ is homeomorphic to a sphere.
One of the most important results on face vectors of simplicial posets is the next result due to Stanley \cite{St} and Masuda \cite{Ma},
which characterize all possible $h$-vectors of simplicial cell spheres.

\begin{theorem}[Stanley, Masuda]
\label{1.1}
Let $h=(h_0,h_1,\dots,h_d) \in \ZZ^{d+1}$.
Then $h$ is the $h$-vector of a $(d-1)$-dimensional simplicial cell sphere if and only if
it satisfies the following conditions:
\begin{itemize}
\item[(1)] $h_0=h_d=1$ and $h_i=h_{d-i}$ for all $i$.
\item[(2)] $h_i \geq 0$ for all $i$.
\item[(3)] if $h_i=0$ for some $1 \leq i \leq d-1$ then $h_0+h_1+ \cdots +h_d$ is even.
\end{itemize} 
\end{theorem}

Theorem \ref{1.1} characterizes the face vectors of simplicial cell spheres.
We say that a poset $P$ is a \textit{simplicial cell decomposition of a topological space $X$}
if $P$ is a simplicial poset such that $\Gamma(P)$ is homeomorphic to $X$.
From topological and combinatorial viewpoints,
it is natural to ask a characterization of face vectors of simplicial cell decompositions of a given topological manifold.
In this paper, we give such a characterization for the product of spheres.

Before stating the result,
we define $h''$-vectors introduced by Novik \cite{No}.
From now on, we fix a field $K$.
For a simplicial poset $P$,
let
$$\beta_i=\beta_i(P)=\dim_K \tilde H_i(P;K)$$
be the $i$th \textit{Betti number of $P$},
where $\tilde H_i(P;K)$ is the $i$th reduced homology group of $P$ (or $\Gamma(P)$) over $K$.
The \textit{$h''$-vector $h''(P)=(h_0'',h_1'',\dots,h_d'')$ of $P$ (over $K$)} is defined by
\begin{eqnarray*}
h_k''(P)= \left\{
\begin{array}{lll}
1, & \mbox{ if } k=0,\\
\displaystyle{h_k- {d \choose k} \left\{\sum_{\ell=1}^{k} (-1)^{\ell -k} \beta_{\ell-1} \right\}},
&\mbox{ if }1\leq k \leq d-1,\smallskip\\
h_d -\sum_{\ell=1}^{d-1}(-1)^{\ell-d}\beta_{\ell-1}=\beta_{d-1},
& \mbox{ if } k=d.
\end{array}
\right.
\end{eqnarray*}
If one knows Betti numbers, then knowing $h(P)$ is equivalent to knowing $h''(P)$.
(Since Betti numbers depend on the characteristic,
$h''$-vectors depend on the characteristic of the base field $K$.)
It was proved by Novik \cite{No} and Novik-Swartz \cite{NS} that
the $h''$-vector of a simplicial cell decomposition of an orientable manifold is symmetric and non-negative
(see section 2).
The main result of this paper is the next result,
which characterizes face vectors of simplicial cell decompositions of the
product of spheres $S^n \times S^m$.

\begin{theorem}
\label{1.2}
Fix integers $n, m \geq 1$.
Let $d=n+m+1$ and $h=(h_0,h_1,\dots,h_d) \in \ZZ^{d+1}$.
There exists a simplicial cell decomposition $P$ of $S^n \times S^m$ with $h''(P)=h$
if and only if $h$ satisfies the conditions (1), (2) and (3) in Theorem \ref{1.1}.
\end{theorem}

The tequnique used in the proof of the above theorem is also applicable to
other classes of simplicial posets.
We characterize all possible face vectors of simplicial cell decompositions
of real projective spaces (Theorem \ref{7.1}) and face vectors of simplicial cell complexes
that are odd dimensional topological manifolds without boundary (Theorem \ref{7.2}).

This paper is organized as follows:
In section 2, we recall known conditions on $h''$-vectors and prove the necessity of Theorem \ref{1.2}.
In section 3--6,
we prove the sufficiency of Theorem \ref{1.2}.
The key idea to prove the sufficiency is a graph theoretical construction of simplicial cell decompositions of
manifolds called crystallizations \cite{FGG}.
In section 7, we discuss face vectors of simplicial cell decompositions of real projective spaces and odd dimensional manifolds.

\section{Necessity conditions of $h''$-vectors}

In this section, we recall some known necessity conditions of $h''$-vectors.

Let $P$ be a $(d-1)$-dimensional simplicial poset.
For an element $\sigma \in P$,
the \textit{link} of $\sigma$ in $P$, denoted by $P_{\geq \sigma}$, is a poset
$$P_{\geq \sigma} = \{ \tau \in P : \tau \geq \sigma\}.$$
It is easy to see that $P_{\geq \sigma}$ is again a simplicial poset with the minimal element $\sigma$.
For $k=0,1,\dots,d$
we write
$$P_k=\{\sigma \in P: \rank \sigma =k\}.$$
An element in $P_1$ is called a \textit{vertex of $P$}
and a maximal element in $P$ is called a \textit{facet of $P$}.
A simplicial poset $P$ is said to be \textit{pure} if all its facets have the same rank.

A $(d-1)$-dimensional simplicial poset $P$ is said to be a \textit{homology sphere (over $K$)}
if, for all $\sigma \in P$, $\beta_i(P_{\geq \sigma})=0$ for all $i \ne d-1-\rank \sigma$
and $\beta_{d-1-\rank \sigma }(P_{\geq \sigma})=1$.
Also, a pure simplicial poset $P$ is said to be a \textit{homology manifold (without boundary)}
if $P_{\geq v}$ is a homology sphere for all vertices $v \in P_1$.
A \textit{simplicial cell homology manifold} is a simplicial poset which is a homology manifold.
From now on, we assume that all homology manifolds are connected.
A simplicial cell homology manifold $P$ is said to be \textit{orientable}
if $\beta_{d-1}(P)=1$.

The next result is crucial for the necessity of Theorem \ref{1.2}.

\begin{theorem}
\label{2.1}
Let $P$ be a $(d-1)$-dimensional simplicial cell homology manifold. Then
\begin{itemize}
\item[(i)](Novik-Swartz) $h_i''(P) \geq 0$ for all $i$.
\item[(ii)](Novik) if $P$ is orientable then $h_i''(P)=h_{d-i}''(P)$ for all $i$.
\end{itemize}
\end{theorem}

The first condition was recently proved by Novik and Swartz \cite[Proposition 6.3 and Theorem 6.4]{NS}.
The second condition was proved by Novik in \cite[Lemma 7.3]{No}
for simplicial complexes.
However, this condition essentially follows from the Dehn--Sommerville equations for homology manifolds and the Poincar\'e duality.
Since both the Dehn--Sommerville equations and the Poincar\'e duality hold for simplicial posets,
(ii) holds for simplicial posets.
See \cite[section 8]{MMP} for Dehn--Sommerville equations for simplicial posets.

Another necessity condition of $h''$-vectors which we need is the following condition.

\begin{lemma}
\label{2.2}
Let $P$ be a $(d-1)$-dimensional orientable simplicial cell homology manifold.
If $h_i''(P)=0$ for some $1 \leq i \leq d-1$ then the number of facets of $P$ is even.
\end{lemma}

For homology spheres,
the above statement was conjectured by Stanley \cite{St}
and proved by Masuda \cite{Ma}.

To prove Lemma \ref{2.2},
we need an algebraic tool, called face rings.
Since
% here is the only place that we need algebra and since 
the proof is essentially the same as the simple proof of Masuda's result given by Miller and Reiner \cite{MR},
we just sketch the proof.
We refer the readers to \cite{St} and \cite{Du} for basic algebraic properties of face rings
and basic algebraic notations on commutative algebra.

Let $P$ be a simplicial poset,
$R=K[x_\sigma: \sigma \in P \setminus \{ \hat 0\}]$
the polynomial ring over a field $K$ in indeterminates indexed by the elements in $P \setminus \{\hat 0\}$
and $S=K[x_v: v \in P_1]$.
The \textit{face ring of $P$} is the quotient ring $K[P]=R/I_P$,
where $I_P$ is the ideal generated by the following elements
\begin{itemize}
\item $x_\sigma x_\tau$ for all pairs $\sigma,\tau \in P$ that have no common upper bounds in $P$.
\item $x_\sigma x_\tau - x_{\sigma \wedge \tau} \sum_\rho x_\rho$,
where the summation runs over the minimal elements among all upper bounds of $\sigma$ and $\tau$,
and where $\sigma \wedge \tau$ is the meet (largest lower bounds) of $\sigma$ and $\tau$.
(We consider $x_{\sigma \wedge \tau}=1$ if $\sigma \wedge \tau=\hat 0$.)
\end{itemize}
It is known that,
by setting $\deg x_\sigma=\rank \sigma$,
$K[P]$ is a $d$-dimensional finitely generated $S$-module whose Hilbert series determines the $f$-vector of $P$,
where $d=\dim P+1$.
See \cite[Proposition 3.8 and Lemma 3.9]{St}.

\begin{proof}[Proof of Lemma \ref{2.2}](Sketch).
Throughout the proof we regard $K[P]$ as an $S$-module.
Let $\theta_1,\dots,\theta_d \in S$ be an l.s.o.p.\ of $K[P]$
(it exists by assuming that $K$ is infinite if necessary)
and $A_P = K[P]/(\theta_1,\dots,\theta_d)K[P]$.
By Schenzel's results (see \cite[Proposition 6.3]{NS}),
we have
\begin{itemize}
\item[(NS1)]
$\dim_K (A_P)_d = h''_d =1$
and $\dim_K (A_P)_k = h''_k + {d \choose k} \beta_{k-1}(P)$ for $k=1,\dots,d-1$,
\end{itemize}
where $(A_P)_k$ is the homogeneous component of $A_P$ of degree $k$.
Since $h''_i(P)=0$, it follows from \cite[Theorem 6.4]{NS} that all elements in $(A_P)_i$
are socle elements, that is, for any $f \in (A_P)_i$ and for any homogeneous polynomial $h \in S$ with $h \not \in K$, we have $f h =0$ in $A_P$.
In particular, for distinct vertices $v_1,v_2,\dots,v_d$ of $P$,
we have
\begin{itemize}
\item[(NS2)] $x_{v_1} x_{v_2} \cdots x_{v_d}=0$ \mbox{ in } $A_P$.
\end{itemize}
For an element $\sigma \in P$,
let $V(\sigma)=\{v \in P_1: v \leq \sigma\}$
be the set of vertices of $\sigma$.
Since $P$ is pure, by the definition of the ideal $I_P$,
$x_{v_1} x_{v_2} \cdots x_{v_d}=\sum_{\sigma \in P_d,\ V(\sigma)=\{v_1,\dots,v_d\}} x_\sigma$
in $K[P]$.
Since $P$ is a pseudomanifold (see section 4 for the definition of pseudomanifolds)
it follows from \cite[Propositions 5 and 6]{MR} that
\begin{itemize}
\item[(MR1)] if $(A_P)_d \ne \{0\}$ then, for any facet $\sigma \in P$, $x_\sigma \ne 0$ in $A_P$.
\item[(MR2)] for all facets $\sigma$ and $\tau$ of $P$ with $V(\sigma)=V(\tau)$,
$x_\sigma = \pm x_\tau$.
\end{itemize}
(NS1) shows that the assumption of (MR1) is satisfied.
Then, for distinct vertices $v_1,\dots,v_d$ of $P$,
since (NS2) says
$\sum_{\sigma \in P_d,\ V(\sigma)=\{v_1,\dots,v_d\}} x_\sigma=0$ in $A_P$,
by (MR1) and (MR2) it follows that
the number of faces $\sigma$ of $P$ with $V(\sigma)=\{v_1,\dots,v_d\}$ is even.
Hence the number of facets of $P$ is even.
\end{proof}

\begin{corollary}
\label{2.3}
Let $P$ be a $(d-1)$-dimensional orientable simplicial cell homology manifold.
If $h''_i(P)=0$ for some $1 \leq i \leq d-1$, then $\sum_{i=0}^d h_i''(P)$
is even.
\end{corollary}

\begin{proof}
By the symmetry of $h''$-vectors, we may assume that $d$ is even.
Since $f_d(P)=\sum_{i=0}^d h_i(P)$ is even
by Lemma \ref{2.2}, it is enough to prove that $\sum_{i=0}^d h_i(P)\equiv \sum_{i=0}^d h''_i(P)$ mod 2.
By the definition of $h''$-vectors,
\begin{eqnarray*}
\sum_{i=0}^d h''_i(P)&=& \sum_{i=0}^d h_i(P) + \left[ \sum_{i=1}^{d-2} \beta_i(P) \left\{ \sum_{l=i+1}^d (-1)^{l-i-1} {d \choose l} \right\}\right]\\
&=& \sum_{i=0}^d h_i(P) +  \sum_{i=1}^{d-2} \beta_i(P) {d -1\choose i}\\
&=& \sum_{i=0}^d h_i(P) +  2 \left\{\sum_{i=1}^{\frac {d-2} 2} \beta_i(P) {d -1\choose i} \right\}
\end{eqnarray*}
as desired,
where we use the Poincar\'e duality $\beta_i=\beta_{d-1-i}$ for the last equality.
\end{proof}

Theorem \ref{2.1} and Corollary {2.3} prove the necessity of Theorem \ref{1.2}.
More precisely,

\begin{theorem}
\label{2.4}
If $P$ is a $(d-1)$-dimensional orientable simplicial cell homology manifold,
then $h''(P)$ satisfies the conditions (1), (2) and (3) in Theorem \ref{1.1}.
\end{theorem}

\section{How to characterize $h''$-vectors}

In this section,
we show that to characterize $h''$-vectors of simplicial cell decompositions
of a manifold $M$,
it is enough to find simplicial cell decompositions of $M$ with minimal $h''$-vectors.
From now on, all manifolds are connected, compact and without boundary. 
In addition, we assume that all manifolds and homeomorphisms are piecewise linear (see \cite{Hu}).

Let $P$ and $Q$ be $(d-1)$-dimensional simplicial posets,
$\sigma \in P_d$ and $\tau \in Q_d$.
The connected sum of $P$ and $Q$ with respect to $\sigma$ and $\tau$ is the simplicial poset,
denoted $P\#_{\sigma,\tau} Q$ (or $P \# Q$ for short),
obtained from $P$ and $Q$ by by removing $\sigma$ and $\tau$ from $P$ and $Q$ and by
identifying $[\hat 0,\sigma] \setminus \{ \sigma\}$ and $[\hat 0,\tau] \setminus \{\tau\}$.
Thus, topologically, $P\#Q$ is obtained by removing $(d-1)$-cells $\sigma$ and $\tau$ from $P$ and $Q$ and gluing them along the boundaries of $\sigma$ and $\tau$.

\begin{lemma}
\label{3.1}
Let $P$ be a $(d-1)$-dimensional orientable simplicial cell homology manifold
and $Q$ a $(d-1)$-dimensional simplicial cell homology manifold.
Then $P \# Q$ is a homology manifold satisfying the following conditions
\begin{itemize}
\item[(i)]
$\beta_i(P\#Q)=\beta_i (P) + \beta_i(Q)$ for $i \ne d-1$ and $\beta_{d-1}(P\#Q)=\beta_{d-1}(Q)$.
\item[(ii)] $h_i''(P \#Q)=h_i''(P) + h_i''(Q)$
for $i \ne 0,d-1$ and $h_{d}''(P\#Q)=h''_d(Q)$.
\end{itemize}
\end{lemma}

\begin{proof}
It is straightforward that $P \# Q$ is a homology manifold.
(i)
follows from a simple Mayer--Vietoris argument.
Observe $f_i(P\#Q)=f_i(P)+f_i(Q)- { d \choose i} $
for $i \ne d$
and $f_d(P\#Q)=f_d(P)+f_d(Q)-2$.
Straightforward computations
show
$h_i(P\#Q)=h_i(P)+h_i(Q)$ for $i \ne d$
and $h_d(P\#Q)=h_d(P)+h_d(Q)-1$.
Then (ii) follows from (i) and the definition of $h''$-vectors.
\end{proof}

Let $M$ be a $(d-1)$-dimensional manifold.
We write $\hh (M)$ for the set of all $h''$-vectors of simplicial cell decompositions of $M$,
where we consider $h''$-vectors over a field of characteristic $2$ if $M$ is non-orientable.
For example, if $M =S^{d-1}$ the $(d-1)$-dimensional sphere,
then $\hh (M)$ is the set of all vectors $h \in \ZZ^{d+1}$
satisfying the conditions (1), (2) and (3) in Theorem \ref{1.1}.

\begin{corollary}
\label{3.2}
With the same notation as above,
if there is a simplicial cell decomposition $P$ of $M$ with $h''(P)=(1,0,\dots,0,1)$,
then $\hh(M)=\hh(S^{d-1})$.
\end{corollary}

\begin{proof}
Since any manifold is an orientable homology manifold over a field of characteristic $2$,
Theorem \ref{2.4} shows $\hh(M) \subset \hh(S^{d-1})$.
We prove the reverse inclusion.
Let $h \in \hh(S^{d-1})$.
There exists a $(d-1)$-dimensional simplicial cell sphere $Q$ with $h''(Q)=h(Q)=h$ by Theorem \ref{1.1}.
Then $P\#Q$ is a simplicial cell decomposition of $M$ with the desired $h''$-vector by Lemma \ref{3.1}.
\end{proof}

Corollary \ref{3.2}
shows that the existence of  a simplicial cell decomposition $P$ of $M$
with $h''(P)=(1,0,\dots,0,1)$ 
induces a characterization of face vectors of simplicial
cell decompositions of $M$.

We define a partial order  $>_P$ on $\hh (S^{d-1})$ by,
for $h,h' \in \hh (S^{d-1})$,
 $h>_P h'$ if $h-h'+(1,0,\dots,0,1) \in \hh(S^{d-1})$.
The proof of Corollary \ref{3.2}
says that, to characterize $h''$-vectors of simplicial cell decompositions of $M$,
it is enough to find all minimal elements of $\hh (M)$ with respect to $>_P$.
This fact suggests the following problems.

\begin{problem}
\label{3.3}
For a given manifold $M$,
find all minimal elements in $\hh(M)$.
\end{problem}

\begin{problem}
\label{3.4}
For which manifold $M$,
$\hh(M)$ possess the unique minimal element?
In particular, for which $M$, one has $(1,0,\dots,0,1) \in \hh(M)$?
\end{problem}

For present, we do not even have an example of a manifold $M$ such that $\hh(M)$ has more than two minimal elements.
Note that if $M$ is a Poincar\'e sphere, then $(1,0,\dots,0,1) \not \in \hh(M)$.
\smallskip

\begin{example}
Figure 1 is a simplicial cell complex that presents $S^1 \times S^1$.
(Identify parallel edges of the square.)
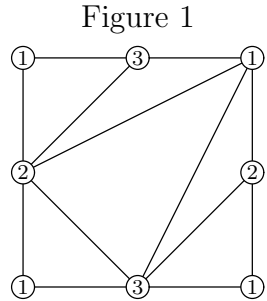
\begin{center}
%WinTpicVersion3.08
\unitlength 0.1in
\begin{picture}( 16.6500, 15.4500)( 19.9500,-16.6000)
% LINE 2 0 3 0
% 8 2400 400 2400 1600 2400 1600 3600 1600 3600 1600 3600 400 3600 400 2400 400
% 
\special{pn 8}%
\special{pa 2400 400}%
\special{pa 2400 1600}%
\special{fp}%
\special{pa 2400 1600}%
\special{pa 3600 1600}%
\special{fp}%
\special{pa 3600 1600}%
\special{pa 3600 400}%
\special{fp}%
\special{pa 3600 400}%
\special{pa 2400 400}%
\special{fp}%
% LINE 2 0 3 0
% 12 2400 1000 3000 400 3000 1600 3000 1600 3000 1600 2400 1000 3000 1600 3600 1000 3000 1600 3600 400 3600 400 2400 1000
% 
\special{pn 8}%
\special{pa 2400 1000}%
\special{pa 3000 400}%
\special{fp}%
\special{pa 3000 1600}%
\special{pa 3000 1600}%
\special{fp}%
\special{pa 3000 1600}%
\special{pa 2400 1000}%
\special{fp}%
\special{pa 3000 1600}%
\special{pa 3600 1000}%
\special{fp}%
\special{pa 3000 1600}%
\special{pa 3600 400}%
\special{fp}%
\special{pa 3600 400}%
\special{pa 2400 1000}%
\special{fp}%
% CIRCLE 2 0 2 0
% 4 2400 400 2400 340 2400 340 2400 340
% 
\special{pn 8}%
\special{sh 0}%
\special{ar 2400 400 60 60  0.0000000 6.2831853}%
% STR 2 0 3 0
% 3 2400 300 2400 400 5 0
% {\tiny 1}
\put(24.0000,-4.0000){\makebox(0,0){{\tiny 1}}}%
% CIRCLE 2 0 2 0
% 4 2400 1000 2400 940 2400 940 2400 940
% 
\special{pn 8}%
\special{sh 0}%
\special{ar 2400 1000 60 60  0.0000000 6.2831853}%
% CIRCLE 2 0 2 0
% 4 2400 1600 2400 1540 2400 1540 2400 1540
% 
\special{pn 8}%
\special{sh 0}%
\special{ar 2400 1600 60 60  0.0000000 6.2831853}%
% CIRCLE 2 0 2 0
% 4 3000 1600 3000 1540 3000 1540 3000 1540
% 
\special{pn 8}%
\special{sh 0}%
\special{ar 3000 1600 60 60  0.0000000 6.2831853}%
% CIRCLE 2 0 2 0
% 4 3600 1600 3600 1540 3600 1540 3600 1540
% 
\special{pn 8}%
\special{sh 0}%
\special{ar 3600 1600 60 60  0.0000000 6.2831853}%
% CIRCLE 2 0 2 0
% 4 3600 1000 3600 940 3600 940 3600 940
% 
\special{pn 8}%
\special{sh 0}%
\special{ar 3600 1000 60 60  0.0000000 6.2831853}%
% CIRCLE 2 0 2 0
% 4 3600 400 3600 340 3600 340 3600 340
% 
\special{pn 8}%
\special{sh 0}%
\special{ar 3600 400 60 60  0.0000000 6.2831853}%
% CIRCLE 2 0 2 0
% 4 3000 400 3000 340 3000 340 3000 340
% 
\special{pn 8}%
\special{sh 0}%
\special{ar 3000 400 60 60  0.0000000 6.2831853}%
% STR 2 0 3 0
% 3 2400 1500 2400 1600 5 0
% {\tiny 1}
\put(24.0000,-16.0000){\makebox(0,0){{\tiny 1}}}%
% STR 2 0 3 0
% 3 3600 1500 3600 1600 5 0
% {\tiny 1}
\put(36.0000,-16.0000){\makebox(0,0){{\tiny 1}}}%
% STR 2 0 3 0
% 3 3600 300 3600 400 5 0
% {\tiny 1}
\put(36.0000,-4.0000){\makebox(0,0){{\tiny 1}}}%
% STR 2 0 3 0
% 3 3000 300 3000 400 5 0
% {\tiny 3}
\put(30.0000,-4.0000){\makebox(0,0){{\tiny 3}}}%
% STR 2 0 3 0
% 3 3000 1500 3000 1600 5 0
% {\tiny 3}
\put(30.0000,-16.0000){\makebox(0,0){{\tiny 3}}}%
% STR 2 0 3 0
% 3 2400 900 2400 1000 5 0
% {\tiny 2}
\put(24.0000,-10.0000){\makebox(0,0){{\tiny 2}}}%
% STR 2 0 3 0
% 3 3600 900 3600 1000 5 0
% {\tiny 2}
\put(36.0000,-10.0000){\makebox(0,0){{\tiny 2}}}%
% STR 2 0 3 0
% 3 3000 100 3000 200 5 0
% Figure 1
\put(30.0000,-2.0000){\makebox(0,0){Figure 1}}%
\end{picture}%
\end{center}
Its $f$-vector is $f=(1,3,9,6)$ and its $h$-vector is $h=(1,0,6,-1)$.
Since $\beta_1(S^1 \times S^1)=2$,
the $h''$-vector is $h''=h-2(0,0,3,-1)=(1,0,0,1)$.
\end{example}

\section{Graphical simplicial posets}

To study problems given in the previous section,
it is important to have a good construction of simplicial cell homology manifolds.
We use graph theoretic approach called \textit{crystallizations}.
In this section, we briefly introduce crystallization theory.
Most statements of this section are not new,
but we rewrite it to adapt the theory to simplicial posets.
A good survey of crystallization theory is \cite{FGG}.

Let $G=(V,E,\phi)$ be a (finite) multi-graph (without loops),
where $V$ is a finite set of vertices,
$E$ is a finite set of edges and $\phi$ is a function that assigns to each edge $e \in E$
a $2$-elements set of vertices $\phi(e)\subset V$.
For an integer $d \geq 1$,
a pair $\Lambda=(G,\gamma)$ of a graph $G=(V,E,\phi)$
and a map $\gamma: E \to [d]=\{1,2,\dots,d\}$ is called a \textit{$d$-colored multi-graph}.
For a $d$-colored multi-graph $\Lambda=(G,\gamma)$ and $S \subset [d]$,
let 
$$E_S=\big\{e \in E: \gamma(e) \in S\big\}$$
and
$$G_S=(V,E_S, \phi_S),$$
where $\phi_S$ is the restriction of $\phi$ to $E_S$.
Thus $G_S$ is the multi-graph whose edges are the edges in $G$ having color $i \in S$.
We say that a $d$-colored multi-graph $\Lambda=(G,\gamma)$ is \textit{admissible} if it satisfies the following conditions:
\begin{itemize}
\item[(a)] $G$ is connected.
\item[(b)] for each $i \in [d]$, $G_{\{i\}}$ is a complete matching on $V$.
In other words, all edges in $G_{\{i\}}$ are vertex-disjoint and every vertex in $V$
is a vertex of an edge of $G_{\{i\}}$.
\end{itemize}
Note that the number of the vertices of $G$ must be even by (b).

For an admissible $d$-colored multi-graph $\Lambda$,
we define a poset $P_\Lambda$ such that its elements are the pairs $(H,S)$ of a connected component $H$
of $G_S$ and a subset $S \subset [d]$
and the order on $P_\Lambda$ is defined by
$$(H,S) \geq (H',S')\ \ \Leftrightarrow \ \ S \subset S' \mbox{ and $H$ is a subgraph of $H'$}.$$
Thus $H$ consists of a single vertex of $G$ if $S=\emptyset$ (since $G_{\emptyset}=(V,\emptyset,\phi_\emptyset)$)
and $H$ consists of a single edge if $S=\{i\}$ since $G_{\{i\}}$ 
is a matching.
Figure 2 is an example of an admissible $3$-colored multi-graph $\Lambda$
and the poset $P_\Lambda$.
(This $P_\Lambda$ is the simplicial cell decomposition of $S^1 \times S^1$
given in Figure 1.)
Many examples of admissible colored graphs that present manifolds can be found in \cite{FGG}.
\smallskip

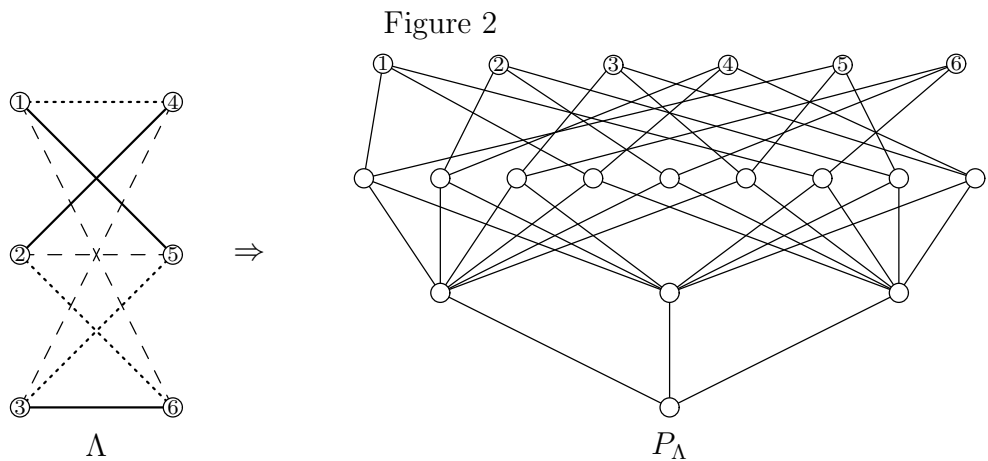
\begin{center}
%WinTpicVersion3.08
\unitlength 0.1in
\begin{picture}( 54.5500, 22.0000)(  3.9500,-23.1500)
% LINE 1 0 3 0
% 6 800 600 1600 1400 1600 600 800 1400 800 2200 1600 2200
% 
\special{pn 13}%
\special{pa 800 600}%
\special{pa 1600 1400}%
\special{fp}%
\special{pa 1600 600}%
\special{pa 800 1400}%
\special{fp}%
\special{pa 800 2200}%
\special{pa 1600 2200}%
\special{fp}%
% LINE 1 2 3 0
% 6 800 600 1600 600 1600 1400 800 2200 800 1400 1600 2200
% 
\special{pn 13}%
\special{pa 800 600}%
\special{pa 1600 600}%
\special{dt 0.045}%
\special{pa 1600 1400}%
\special{pa 800 2200}%
\special{dt 0.045}%
\special{pa 800 1400}%
\special{pa 1600 2200}%
\special{dt 0.045}%
% LINE 2 1 3 0
% 6 800 600 1600 2200 800 2200 1600 600 1600 1400 800 1400
% 
\special{pn 8}%
\special{pa 800 600}%
\special{pa 1600 2200}%
\special{da 0.070}%
\special{pa 800 2200}%
\special{pa 1600 600}%
\special{da 0.070}%
\special{pa 1600 1400}%
\special{pa 800 1400}%
\special{da 0.070}%
% STR 2 0 3 0
% 3 2000 1300 2000 1400 5 0
% $\Rightarrow$
\put(20.0000,-14.0000){\makebox(0,0){$\Rightarrow$}}%
% CIRCLE 2 0 2 0
% 4 800 600 800 550 800 550 800 550
% 
\special{pn 8}%
\special{sh 0}%
\special{ar 800 600 50 50  0.0000000 6.2831853}%
% CIRCLE 2 0 2 0
% 4 800 1400 800 1350 800 1350 800 1350
% 
\special{pn 8}%
\special{sh 0}%
\special{ar 800 1400 50 50  0.0000000 6.2831853}%
% CIRCLE 2 0 2 0
% 4 800 2200 800 2150 800 2150 800 2150
% 
\special{pn 8}%
\special{sh 0}%
\special{ar 800 2200 50 50  0.0000000 6.2831853}%
% CIRCLE 2 0 2 0
% 4 1600 2200 1600 2150 1600 2150 1600 2150
% 
\special{pn 8}%
\special{sh 0}%
\special{ar 1600 2200 50 50  0.0000000 6.2831853}%
% CIRCLE 2 0 2 0
% 4 1600 1400 1600 1350 1600 1350 1600 1350
% 
\special{pn 8}%
\special{sh 0}%
\special{ar 1600 1400 50 50  0.0000000 6.2831853}%
% CIRCLE 2 0 2 0
% 4 1600 600 1600 550 1600 550 1600 550
% 
\special{pn 8}%
\special{sh 0}%
\special{ar 1600 600 50 50  0.0000000 6.2831853}%
% STR 2 0 3 0
% 3 800 500 800 600 5 0
% {\tiny 1}
\put(8.0000,-6.0000){\makebox(0,0){{\tiny 1}}}%
% STR 2 0 3 0
% 3 800 1300 800 1400 5 0
% {\tiny 2}
\put(8.0000,-14.0000){\makebox(0,0){{\tiny 2}}}%
% STR 2 0 3 0
% 3 800 2100 800 2200 5 0
% {\tiny 3}
\put(8.0000,-22.0000){\makebox(0,0){{\tiny 3}}}%
% STR 2 0 3 0
% 3 1600 2100 1600 2200 5 0
% {\tiny 6}
\put(16.0000,-22.0000){\makebox(0,0){{\tiny 6}}}%
% STR 2 0 3 0
% 3 1600 1300 1600 1400 5 0
% {\tiny 5}
\put(16.0000,-14.0000){\makebox(0,0){{\tiny 5}}}%
% STR 2 0 3 0
% 3 1600 500 1600 600 5 0
% {\tiny 4}
\put(16.0000,-6.0000){\makebox(0,0){{\tiny 4}}}%
% LINE 2 0 3 0
% 6 2700 400 2600 1000 2700 400 3800 1000 3300 400 3300 400
% 
\special{pn 8}%
\special{pa 2700 400}%
\special{pa 2600 1000}%
\special{fp}%
\special{pa 2700 400}%
\special{pa 3800 1000}%
\special{fp}%
\special{pa 3300 400}%
\special{pa 3300 400}%
\special{fp}%
% LINE 2 0 3 0
% 2 3300 400 3000 1000
% 
\special{pn 8}%
\special{pa 3300 400}%
\special{pa 3000 1000}%
\special{fp}%
% LINE 2 0 3 0
% 2 2600 1000 5100 400
% 
\special{pn 8}%
\special{pa 2600 1000}%
\special{pa 5100 400}%
\special{fp}%
% LINE 2 0 3 0
% 2 3000 1000 4500 400
% 
\special{pn 8}%
\special{pa 3000 1000}%
\special{pa 4500 400}%
\special{fp}%
% LINE 2 0 3 0
% 10 3400 1000 3900 400 5700 400 3400 1000 3800 1000 4500 400 4200 1000 3300 400 5700 400 4200 1000
% 
\special{pn 8}%
\special{pa 3400 1000}%
\special{pa 3900 400}%
\special{fp}%
\special{pa 5700 400}%
\special{pa 3400 1000}%
\special{fp}%
\special{pa 3800 1000}%
\special{pa 4500 400}%
\special{fp}%
\special{pa 4200 1000}%
\special{pa 3300 400}%
\special{fp}%
\special{pa 5700 400}%
\special{pa 4200 1000}%
\special{fp}%
% LINE 2 0 3 0
% 4 4600 1000 3900 400 5100 400 4600 1000
% 
\special{pn 8}%
\special{pa 4600 1000}%
\special{pa 3900 400}%
\special{fp}%
\special{pa 5100 400}%
\special{pa 4600 1000}%
\special{fp}%
% LINE 2 0 3 0
% 4 2700 400 5000 1000 5000 1000 5700 400
% 
\special{pn 8}%
\special{pa 2700 400}%
\special{pa 5000 1000}%
\special{fp}%
\special{pa 5000 1000}%
\special{pa 5700 400}%
\special{fp}%
% CIRCLE 2 0 2 0
% 4 2700 400 2700 350 2700 350 2700 350
% 
\special{pn 8}%
\special{sh 0}%
\special{ar 2700 400 50 50  0.0000000 6.2831853}%
% CIRCLE 2 0 2 0
% 4 5700 400 5700 350 5700 350 5700 350
% 
\special{pn 8}%
\special{sh 0}%
\special{ar 5700 400 50 50  0.0000000 6.2831853}%
% STR 2 0 3 0
% 3 5700 300 5700 400 5 0
% {\tiny 6}
\put(57.0000,-4.0000){\makebox(0,0){{\tiny 6}}}%
% STR 2 0 3 0
% 3 2700 300 2700 400 5 0
% {\tiny 1}
\put(27.0000,-4.0000){\makebox(0,0){{\tiny 1}}}%
% STR 2 0 3 0
% 3 3000 100 3000 200 5 0
% Figure 2
\put(30.0000,-2.0000){\makebox(0,0){Figure 2}}%
% STR 2 0 3 0
% 3 1200 2300 1200 2400 5 0
% $\Lambda$
\put(12.0000,-24.0000){\makebox(0,0){$\Lambda$}}%
% STR 2 0 3 0
% 3 4200 2300 4200 2400 5 0
% $P_\Lambda$
\put(42.0000,-24.0000){\makebox(0,0){$P_\Lambda$}}%
% LINE 2 0 3 0
% 4 5400 1000 3300 400 5100 400 5400 1000
% 
\special{pn 8}%
\special{pa 5400 1000}%
\special{pa 3300 400}%
\special{fp}%
\special{pa 5100 400}%
\special{pa 5400 1000}%
\special{fp}%
% LINE 2 0 3 0
% 4 5800 1000 3900 400 4500 400 5800 1000
% 
\special{pn 8}%
\special{pa 5800 1000}%
\special{pa 3900 400}%
\special{fp}%
\special{pa 4500 400}%
\special{pa 5800 1000}%
\special{fp}%
% CIRCLE 2 0 2 0
% 4 3305 405 3305 355 3305 355 3305 355
% 
\special{pn 8}%
\special{sh 0}%
\special{ar 3306 406 50 50  0.0000000 6.2831853}%
% STR 2 0 3 0
% 3 3305 305 3305 405 5 0
% {\tiny 2}
\put(33.0500,-4.0500){\makebox(0,0){{\tiny 2}}}%
% CIRCLE 2 0 2 0
% 4 3905 405 3905 355 3905 355 3905 355
% 
\special{pn 8}%
\special{sh 0}%
\special{ar 3906 406 50 50  0.0000000 6.2831853}%
% STR 2 0 3 0
% 3 3905 305 3905 405 5 0
% {\tiny 3}
\put(39.0500,-4.0500){\makebox(0,0){{\tiny 3}}}%
% CIRCLE 2 0 2 0
% 4 4505 405 4505 355 4505 355 4505 355
% 
\special{pn 8}%
\special{sh 0}%
\special{ar 4506 406 50 50  0.0000000 6.2831853}%
% STR 2 0 3 0
% 3 4505 305 4505 405 5 0
% {\tiny 4}
\put(45.0500,-4.0500){\makebox(0,0){{\tiny 4}}}%
% CIRCLE 2 0 2 0
% 4 5105 405 5105 355 5105 355 5105 355
% 
\special{pn 8}%
\special{sh 0}%
\special{ar 5106 406 50 50  0.0000000 6.2831853}%
% STR 2 0 3 0
% 3 5105 305 5105 405 5 0
% {\tiny 5}
\put(51.0500,-4.0500){\makebox(0,0){{\tiny 5}}}%
% LINE 2 0 3 0
% 6 3000 1600 4200 2200 4200 1600 4200 2200 4200 2200 5400 1600
% 
\special{pn 8}%
\special{pa 3000 1600}%
\special{pa 4200 2200}%
\special{fp}%
\special{pa 4200 1600}%
\special{pa 4200 2200}%
\special{fp}%
\special{pa 4200 2200}%
\special{pa 5400 1600}%
\special{fp}%
% LINE 2 0 3 0
% 24 2600 1000 3000 1600 4200 1600 2600 1000 3000 1000 3000 1600 3000 1000 4200 1600 3400 1000 4200 1600 3000 1600 3400 1000 3800 1000 3000 1600 3800 1000 5400 1600 5400 1600 4200 1000 3000 1600 4200 1000 4600 1000 5400 1600 4600 1000 3000 1600
% 
\special{pn 8}%
\special{pa 2600 1000}%
\special{pa 3000 1600}%
\special{fp}%
\special{pa 4200 1600}%
\special{pa 2600 1000}%
\special{fp}%
\special{pa 3000 1000}%
\special{pa 3000 1600}%
\special{fp}%
\special{pa 3000 1000}%
\special{pa 4200 1600}%
\special{fp}%
\special{pa 3400 1000}%
\special{pa 4200 1600}%
\special{fp}%
\special{pa 3000 1600}%
\special{pa 3400 1000}%
\special{fp}%
\special{pa 3800 1000}%
\special{pa 3000 1600}%
\special{fp}%
\special{pa 3800 1000}%
\special{pa 5400 1600}%
\special{fp}%
\special{pa 5400 1600}%
\special{pa 4200 1000}%
\special{fp}%
\special{pa 3000 1600}%
\special{pa 4200 1000}%
\special{fp}%
\special{pa 4600 1000}%
\special{pa 5400 1600}%
\special{fp}%
\special{pa 4600 1000}%
\special{pa 3000 1600}%
\special{fp}%
% LINE 2 0 3 0
% 12 5400 1600 5800 1000 5800 1000 4200 1600 4200 1600 5000 1000 5000 1000 5400 1600 5400 1000 5400 1600 5400 1000 4200 1600
% 
\special{pn 8}%
\special{pa 5400 1600}%
\special{pa 5800 1000}%
\special{fp}%
\special{pa 5800 1000}%
\special{pa 4200 1600}%
\special{fp}%
\special{pa 4200 1600}%
\special{pa 5000 1000}%
\special{fp}%
\special{pa 5000 1000}%
\special{pa 5400 1600}%
\special{fp}%
\special{pa 5400 1000}%
\special{pa 5400 1600}%
\special{fp}%
\special{pa 5400 1000}%
\special{pa 4200 1600}%
\special{fp}%
% CIRCLE 2 0 2 0
% 4 5800 1000 5800 950 5800 950 5800 950
% 
\special{pn 8}%
\special{sh 0}%
\special{ar 5800 1000 50 50  0.0000000 6.2831853}%
% CIRCLE 2 0 2 0
% 4 5400 1000 5400 950 5400 950 5400 950
% 
\special{pn 8}%
\special{sh 0}%
\special{ar 5400 1000 50 50  0.0000000 6.2831853}%
% CIRCLE 2 0 2 0
% 4 5000 1000 5000 950 5000 950 5000 950
% 
\special{pn 8}%
\special{sh 0}%
\special{ar 5000 1000 50 50  0.0000000 6.2831853}%
% CIRCLE 2 0 2 0
% 4 4600 1000 4600 950 4600 950 4600 950
% 
\special{pn 8}%
\special{sh 0}%
\special{ar 4600 1000 50 50  0.0000000 6.2831853}%
% CIRCLE 2 0 2 0
% 4 4200 1000 4200 950 4200 950 4200 950
% 
\special{pn 8}%
\special{sh 0}%
\special{ar 4200 1000 50 50  0.0000000 6.2831853}%
% CIRCLE 2 0 2 0
% 4 3800 1000 3800 950 3800 950 3800 950
% 
\special{pn 8}%
\special{sh 0}%
\special{ar 3800 1000 50 50  0.0000000 6.2831853}%
% CIRCLE 2 0 2 0
% 4 3400 1000 3400 950 3400 950 3400 950
% 
\special{pn 8}%
\special{sh 0}%
\special{ar 3400 1000 50 50  0.0000000 6.2831853}%
% CIRCLE 2 0 2 0
% 4 3000 1000 3000 950 3000 950 3000 950
% 
\special{pn 8}%
\special{sh 0}%
\special{ar 3000 1000 50 50  0.0000000 6.2831853}%
% CIRCLE 2 0 2 0
% 4 2600 1000 2600 950 2600 950 2600 950
% 
\special{pn 8}%
\special{sh 0}%
\special{ar 2600 1000 50 50  0.0000000 6.2831853}%
% CIRCLE 2 0 2 0
% 4 3000 1600 3000 1550 3000 1550 3000 1550
% 
\special{pn 8}%
\special{sh 0}%
\special{ar 3000 1600 50 50  0.0000000 6.2831853}%
% CIRCLE 2 0 2 0
% 4 4200 1600 4200 1550 4200 1550 4200 1550
% 
\special{pn 8}%
\special{sh 0}%
\special{ar 4200 1600 50 50  0.0000000 6.2831853}%
% CIRCLE 2 0 2 0
% 4 5400 1600 5400 1550 5400 1550 5400 1550
% 
\special{pn 8}%
\special{sh 0}%
\special{ar 5400 1600 50 50  0.0000000 6.2831853}%
% CIRCLE 2 0 2 0
% 4 4200 2200 4200 2150 4200 2150 4200 2150
% 
\special{pn 8}%
\special{sh 0}%
\special{ar 4200 2200 50 50  0.0000000 6.2831853}%
\end{picture}%
\end{center}
\bigskip

It is not hard to see that $P_\Lambda$ is simplicial.
Indeed, for $(H,S) \in P_\Lambda$ and for every $S' \supset S$,
there is the unique connected component $H'$ of $G_{S'}$ which contains $H$.
Also, it is clear that $(G,[d])$ is the unique minimal element of $P_\Lambda$.
These facts show that the interval $[(G,[d]), (H,S)]$
is isomorphic to the poset of the set of all subsets of $[d] \setminus S$
ordered by inclusion. Hence $P_\Lambda$ is simplicial.
Moreover, $\rank (H,S)=d-\# S$ and $\dim P_\Lambda = d-1$,
where $\# X$ is the cardinality of a finite set $X$.

We say that a simplicial poset $P$ is \textit{graphical} if there exists an admissible colored multi-graph $\Lambda$ such that $P$
is isomorphic to $P_\Lambda$ as posets.
In the rest of this section, we study which simplicial posets are graphical.

A $(d-1)$-dimensional simplicial poset is said to be a \textit{pseudomanifold (without boundary)} if $P$ satisfies the following conditions:
\begin{itemize}
\item[(i)] $P$ is \textit{pure}.
\item[(ii)] every element $\sigma \in P_{d-1}$ is covered by exactly two elements in $P_d$.
\item[(iii)] $P$ is \textit{strongly connected}. In other words, for all $\sigma, \tau \in P_d$,
there is a sequence $\sigma =\sigma_1,\sigma_2,\dots,\sigma_p=\tau$ of elements of $P_d$
such that the meet $\sigma_i \wedge \sigma_{i+1}$ of $\sigma_i$ and $\sigma_{i+1}$ has rank $d-1$ for all $i$. 
\end{itemize}
Moreover, a pseudomanifold $P$ is said to be \textit{normal} if for every $\sigma \in P$
with $\rank \sigma \leq d-2$,
its link $P_{\geq \sigma}$ is connected (as CW-complexes).
It is not hard to see that a link of a normal pseudomanifold is again a normal pseudomanifold
(see \cite[p.\ 331]{BD}).

Observe that if $P_\Lambda$ is a graphical simplicial poset and $(H,S) \in P_\Lambda$
then the link $(P_\Lambda)_{\geq (H,S)}$ is the graphical simplicial poset $P_{\Lambda'}$
of the graph $\Lambda'=(H,\gamma_H)$,
where $\gamma_H$ is the restriction of the coloring map of $\Lambda$ to $H$.
Then it is straightforward that any graphical simplicial poset is a normal pseudomanifold.

Another combinatorial property of graphical simplicial posets is the fact that it has a nice
coloring on their vertices.
Let $P$ be a simplicial poset.
Recall that $V(\sigma)=\{v \in P_1: v \leq \sigma\}$, where $\sigma \in P$,
is the set of vertices of $\sigma$.
We say that $P$ is \textit{$d$-colored} if there exists a map $\psi: P_1 \to [d]$
such that, for every $\sigma \in P$, $\psi(u) \ne \psi(v)$ for all $u,v \in V(\sigma)$ with $u \ne v$.
Recall that any vertex of a $(d-1)$-dimensional graphical simplicial poset $P_\Lambda$ is an element of the form
$(H,[d] \setminus \{i\})$.
Then $P_\Lambda$ is $d$-colored
by defining $\psi((H,[d] \setminus \{i\}))=i$.

\begin{proposition}
A $(d-1)$-dimensional simplicial poset $P$ is graphical if and only if
$P$ is a $d$-colored normal pseudomanifold.
\end{proposition}

\begin{proof}
We already proved the ``only if" part.
We prove that if $P$ is a $d$-colored normal pseudomanifold then there exists an admissible $d$-colored
multi-graph $\Lambda=(G,\gamma)$ such that $P$ is isomorphic to $P_\Lambda$.

Let $G=(V,E,\phi)$ be the multi-graph such that
$V=P_d$, $E=P_{d-1}$ and, for any $\sigma \in E$,
$\phi(\sigma)$ is the set of the elements in $P_d$ which cover $\sigma$
(this is well-defined by the condition (ii) of pseudomanifolds).
Let $\psi: P_1 \to [d]$ be a coloring map of $P$.
Define a $d$-colored multi-graph $\Lambda=(G,\gamma)$ by setting
$\gamma(\sigma)$ to be the integer $i \in [d]$ such that $i \not \in \{ \psi(v): v \in V(\sigma)\}$, where $\sigma \in E = P_{d-1}$.

Since $P$ is strongly connected, the graph $G$ is connected.
Also, each $\Lambda_{\{i\}}$ is a complete matching since for every vertex $\sigma \in V=P_d$ and $i \in [d]$,
there is the unique $\tau \in P_{d-1}$ with $\sigma \geq \tau$ and $\{\psi(v): v \in V(\tau)\}=[d]\setminus \{i\}$.
Hence $\Lambda$ is admissible.

We claim that $P$ is isomorphic to $P_\Lambda$ as posets.
Let $\sigma \in P$ and $S =\{ \psi(v): v \in V(\sigma)\}$.
Choose a facet $\sigma' \in P_d$ with $\sigma' \geq \sigma$.
Then there exists the unique connected component $H$ of $\Lambda_{[d]\setminus S}$
which contains the vertex $\sigma' \in V=P_d$ of $\Lambda$.
We define
$$\Phi(\sigma)=(H,[d]\setminus S).$$
This $\Phi(\sigma)$ do not depend on the choice of $\sigma' \in P_d$
with $\sigma' \geq \sigma$.
Indeed, for $\sigma'' \in P_d$ with $\sigma'' \geq \sigma$,
since $P_{\geq \sigma}$ is strongly connected there exist edges $\tau_1,\dots,\tau_m \in (P_{\geq \sigma})_{d-1 -\rank \sigma} \subset E$ which connect $\sigma''$ and $\sigma'$.
Since $\tau_i \in P_{\geq \sigma}$,
$\gamma(\tau_i) \not \in S$ for all $i$.
Hence $\tau_1,\dots,\tau_m$ are edges in $\Lambda_{[d]\setminus S}$
and therefore
$\sigma'$ and $\sigma''$ are in the same connected component of $\Lambda_{[d]\setminus S}$.

We claim that $\Phi$ is an order-preserving bijection.
It is clear that $\Phi: P \to P_\Lambda$ is order preserving.
It remains to prove that $\Phi$ is a bijection.

Let $(H,S) \in P_\Lambda$.
Choose a facet $\tau$ which is a vertex of $H$.
Since $P$ is $d$-colored, there is $\sigma \in P$ with $\sigma \leq \tau$
and with $\{\psi(v): v \in V(\sigma)\}=[d]\setminus S$.
By the definition of $\Phi$,
we have $\Phi(\sigma)=(H,S)$.
Hence $\Phi$ is surjective.

Let $\sigma, \tau \in P$ such that $\Phi(\sigma)=\Phi(\tau)=(H,[d]\setminus S)$.
We prove $\sigma=\tau$.
Observe that, for any facet $\rho \in P_d$ and $T \subset [d]$,
there exists the unique element $\rho' \leq \rho$ in $P$ with $\{ \psi(v): v \in V(\rho')\}=T$.
This fact shows that, for any facet $\rho \geq \sigma$ and an edge $e \in P_{d-1}$
in $\Lambda_{[d]\setminus S}$ with $\rho \in \phi(e)$,
one has $e \geq \sigma$.
Hence the vertices of $H$ are the facets of $P_{\geq \sigma}$,
and therefore there is a facet $\rho \in P_{d-1}$ satisfying $\rho \geq \sigma$ and $\rho \geq \tau$.
Since $\Phi(\sigma)=\Phi(\tau)$,
$\{\psi(v): v \in V(\sigma)\}=\{\psi(v): v \in V(\tau)\}$.
Then $\sigma$ and $\tau$ have a common upper bound $\rho$
and has the same color. This fact implies $\sigma = \tau$.
Hence $\Phi$ is injective.
\end{proof}

Since homology manifolds are normal pseudomanifolds,
we have the following result.

\begin{corollary}
A $(d-1)$-dimensional simplicial cell homology manifold is graphical if and only if
it is $d$-colored.
\end{corollary}

Finally, we explain what crystallizations are.
Given a $(d-1)$-dimensional manifold $M$,
an admissible $d$-colored multi-graph $\Lambda$ is called a \textit{crystallization of $M$}
if (the barycentric subdivision of) the simplicial cell complex $\Gamma(P_\Lambda)$
is homeomorphic to $M$
and, for every $i \in [d]$,
$\Lambda_{[d]\setminus \{i\}}$ is connected.
Since the latter condition is equivalent to $f_1(P_\Lambda)=d$,
which is also equivalent to $h_1''(P_\Lambda)=h_1(P_\Lambda)=0$,
and since any $(d-1)$-dimensional simplicial poset with $d$ vertices are $d$-colored,
considering crystallizations of $M$
is almost equivalent to considering simplicial cell decompositions of $M$ with $d$ vertices.

%\begin{remark}
%\label{4.4}
%By considering multi-graphs with loops,
%one can deal manifolds with boundary.
%See \cite{CG}.
%\end{remark}

%\begin{remark}
%\label{4.5}
%For simplicial complexes,
%coloring condition gives strong restrictions to its face vectors (see \cite{St}).
%However, for simplicial posets, coloring condition seem to give small restrictions to their face vectors.
%If a simplicial cell homology manifold is graphical then the number of its facets must be even.
%This seems to be the only restriction which coloring condition forces to face vectors of simplicial cell homology manifolds.
%Indeed, if follows from the Stanley's proof of the sufficiency of Theorem \ref{1.1},
%it is not hard to see that $h=(h_0,\dots,h_d)$ is the $h$-vector of a $(d-1)$-dimensional $d$-colored boolean sphere if and only if it is symmetric, $h_0=h_d=1$ and $h_0+\cdots+h_d$ is even.
%\end{remark}

\section{Construction of a graph that presents $S^n \times S^m$.}

%To simplify the argument,
%we identify a simplicial poset $P$
%and the simplicial cell complex $\Gamma(P)$.
Let $\Lambda_1$ and $\Lambda_2$ be admissible colored multi-graphs
such that $P_{\Lambda_1}$ and $P_{\Lambda_2}$ are simplicial cell decompositions
of manifolds $M_1$ and $M_2$.
Gagliardi and Grasseli \cite{GG} gave a way to construct an admissible
colored multi-graph $\Lambda$ that gives a simplicial cell decomposition of $M_1 \times M_2$
from $\Lambda_1$ and $\Lambda_2$.
Cristofori \cite{C} studied their construction for the products of spheres.
In this section, we recall this construction of a graph that presents $S^n \times S^m$.

We first recall a standard triangulation of the product of simplexes.
We just list the known facts and do not give a proof.
See \cite{GG} for the details.

Let $\sigma$ be an $n$-dimensional (geometric) simplex in $\mathbb{R}^n$
with vertices $v_0,v_1,\dots,v_n$
and $\tau$ an $m$-dimensional (geometric) simplex in $\mathbb{R}^m$ with vertices 
$u_0,u_1,\dots,u_m$.
The product of $\sigma$ and $\tau$ is the polytope
$$\sigma \times \tau =\{ (x, y) : x \in \sigma \mbox{ and } y \in \tau\} \subset \RR^{n+m}.$$
Then the set of the vertices of $\sigma \times \tau$
is
$$W=\big\{(v_i,u_j): 0 \leq i \leq n,\ 0 \leq j \leq m\big\}.$$
To simplify the notation, we write $w_{ij}=(v_i,u_j)$.

Let $\pd(\sigma \times \tau)$ be the abstract simplicial complex on the vertex set $W$
(that is, a family of subsets of $W$ closed under inclusion)
defined by
$$\pd(\sigma \times \tau)
=\big\{\{w_{i_0j_0},w_{i_1j_1},\dots,w_{i_pj_p}\}: i_0 \leq i_1 \leq \cdots \leq i_p,\ j_0 \leq j_1 \leq \cdots \leq j_p \big\}.$$
In particular, the facets of $\pd(\sigma \times \tau)$
are the sets of the form
$$F(i_0,i_1,\dots,i_{n+m};j_0,j_1,\dots,j_{n+m})
=\{w_{i_0j_0},w_{i_1j_1},\dots,w_{i_{n+m}j_{n+m}}\},
$$
where $(i_0,j_0)=(0,0)$ and where $(i_{l+1},j_{l+1})$ is either $(i_l+1,j_l)$
or $(i_l,j_l+1)$ for all $l$.
Then, by taking a convex hull of each face,
$\pd(\sigma \times \tau)$ gives a triangulation of $\sigma \times \tau$.
Also, $\pd (\sigma \times \tau)$ satisfies the following conditions.
\begin{itemize}
\item $\pd(\sigma \times \tau)$ is $(n+m+1)$-colored by the coloring map
\begin{eqnarray}
\label{5.0}
\begin{array}{cccc}
\psi :& W & \to & [n+m+1],\\
&w_{ij} &\to& \psi(w_{ij})=i+j+1.
\end{array}
\end{eqnarray}
\item
The boundary of $\pd(\sigma \times \tau)$ is generated by the following faces
\begin{eqnarray}
\label{5.1}
&&F(i_0,\dots,i_{n+m};j_0,\dots,j_{n+m})\! \setminus\!  \{w_{i_lj_l}\} \mbox{ such that $\{i_0,\dots,i_{n+m}\}\!  \setminus\!  \{i_l\}\!  \ne\!  \{0,\dots,n\}$,}\\
\label{5.2}
&&F(i_0,\dots,i_{n+m};j_0,\dots,j_{n+m})\!  \setminus\!  \{w_{i_lj_l}\} \mbox{ such that $\{j_0,\dots,j_{n+m}\}\!  \setminus\!  \{j_l\}\!  \ne\!  \{0,\dots,m\}$}.
\end{eqnarray}
Moreover, the convex hull of a face (\ref{5.1}) belongs to $\partial \sigma \times \tau$
and that of a face (\ref{5.2}) belongs to $\sigma \times \partial \tau$.
\end{itemize}

\begin{definition}
\label{5-0}
For any $n$-subset $S \subset [n+m]$,
we associate a facet $F(S)$ of $\pd(\sigma \times \tau)$
as follows:
We define $(0,0)=(i_0,j_0),(i_1,j_1),\dots,(i_{n+m},j_{n+m})=(n,m)$
by
\begin{eqnarray*}
(i_l,j_l) = \left\{
\begin{array}{ccc}
(i_{l-1},j_{l-1})+(1,0), & \mbox{ if } l \in S,\\
(i_{l-1},j_{l-1})+(0,1), & \mbox{ if } l \not \in S,\\
\end{array}
\right.
\end{eqnarray*}
for $\ell =1,2,\dots,n+m$
and let
$$F(S)=F(i_0,\dots,i_{n+m};j_0,\dots,j_{n+m}).$$
Then $\{ F(S): S \subset [n+m],\ \# S=n\}$
is the set of facets of $\pd(\sigma \times \tau)$.
\end{definition}

%It is sometimes convenient to identify
%$F(S)=F(i_0,\dots,i_{n+m};j_0,\dots,j_{n+m})$ with
%the lattice path
%$(i_0,j_0)\to (i_1,j_1)\to \dots \to(i_{n+m},j_{n+m})$.
%See Figure 1.
%\medskip

%\begin{center}
%\input{Fig1}
%\end{center}
%\medskip

%\noindent
%Note that $F(\{i_1,\dots,i_n\})$ corresponds to the lattice path
%whose $i_l$-step is horizontal move for $l=1,2,\dots,n$
%and other steps are vertical moves.

Now we consider the product of spheres.
Let $\sigma_1$ and $\sigma_2$ be $n$-dimensional simplexes
and $\tau_1$ and $\tau_2$ $m$-dimensional simplexes.
Let
\begin{eqnarray*}
A=\pd(\sigma_1 \times \tau_1),&&
B=\pd(\sigma_2 \times \tau_1),\\
C=\pd(\sigma_1 \times \tau_2), &&
D=\pd(\sigma_2 \times \tau_2).
\end{eqnarray*}
Then we obtain a simplicial cell decomposition of $S^n \times S^m$ by identifying
$(\partial \sigma_1)\times \tau_1$ in $A$ and
$(\partial \sigma_2)\times \tau_1$ in $B$,
$(\partial \sigma_1)\times \tau_2$ in $C$ and 
$(\partial \sigma_2)\times \tau_2$ in $D$,
$ \sigma_1 \times (\partial \tau_1)$ in $A$ and 
$\sigma_1 \times (\partial \tau_2)$ in $C$,
$\sigma_2 \times (\partial \tau_1)$ in $B$ and 
$\sigma_2 \times (\partial \tau_2)$ in $D$.
In particular,
by identifying the same types of faces described in (\ref{5.1}) and (\ref{5.2}),
we can construct such a simplicial cell decomposition
in a unique way so that it is $(n+m+1)$-colored by the coloring map (\ref{5.0}).

Let $P(n,m)$ be the simplicial cell decomposition of $S^n \times S^m$
obtained by the above construction.
Since $P(n,m)$ is $(n+m+1)$-colored, it is graphical.
Let $\Lambda(n,m)$ be the admissible $(n+m+1)$-colored graph with $P(n,m) \cong P_{\Lambda(n,m)}$.

For an $n$-subset $S \subset [n+m]$,
let $A(S)$ be the facet of $\pd(\sigma_1 \times \tau_1)$
defined in the same way as in Definition \ref{5-0}.
Also, we define $B(S)$, $C(S)$ and $D(S)$ similarly.
We may consider that these $A(S),B(S),C(S)$ and $D(S)$ are the vertices of $\Lambda(n,m)$.
By (\ref{5.1}) and (\ref{5.2}),
any edge of $\Lambda(n,m)$ is one of the following edges:
%
% (E1), (E2) and (E3)
%are the all edges of $\Lambda(n,m)$:
\begin{itemize}
\item[(E1)] an edge of color $k \in [n+m+1]$ whose vertices are $A(S)$ and $B(S)$
(or $C(S)$ and $D(S)$) such that $\{k-1,k\} \cap S = \emptyset$;
\item[(E2)] an edge of color $k \in [n+m+1]$ whose vertices are $A(S)$ and $C(S)$
(or $B(S)$ and $D(S)$) such that $\{k-1,k\} \subset S$;
\item[(E3)] an edge of color $2 \leq k \leq n+m$ whose vertices are $A(S)$ and $A((S \setminus \{k\}) \cup \{k-1\})$
such that $k \in S$ and $k-1 \not \in S$
(and the same type of edges for $B(-),C(-)$ and $D(-)$);
\end{itemize}
where we consider $\{k,k-1\} =\{1\}$ if $k=1$ and $\{k,k-1\}=\{n+m\}$ if $k=n+m+1$.

\begin{example}
The following Figure 3 is a part of the graph $\Lambda(2,2)$
(the whose graph can be found in \cite[p.\ 567]{GG}).
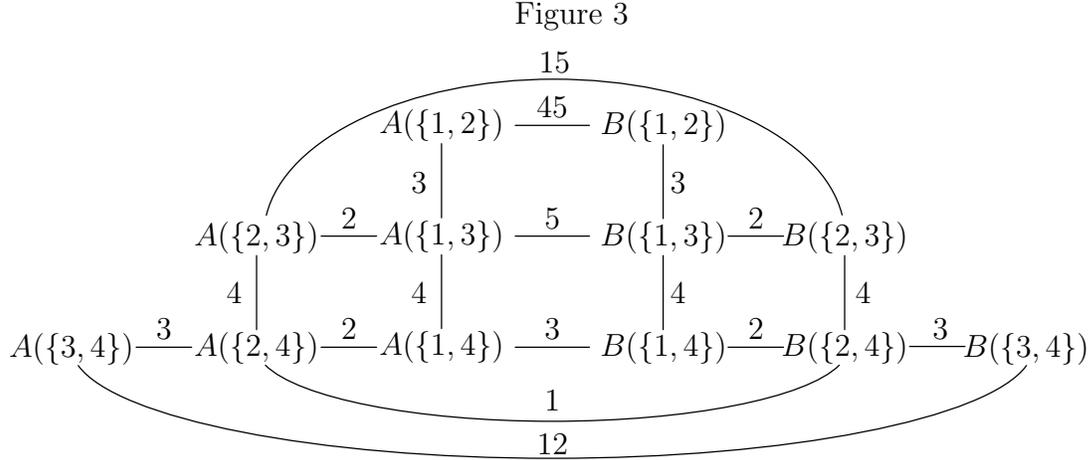
\begin{center}
%WinTpicVersion3.08
\unitlength 0.1in
\begin{picture}( 55.3800, 23.9900)(  3.4300,-29.1400)
% STR 2 0 3 0
% 3 1551 1660 1551 1757 5 0
% $A(\{2,3\})$
\put(15.5100,-17.5700){\makebox(0,0){$A(\{2,3\})$}}%
% STR 2 0 3 0
% 3 2519 1075 2519 1172 5 0
% $A(\{1,2\})$
\put(25.1900,-11.7200){\makebox(0,0){$A(\{1,2\})$}}%
% STR 2 0 3 0
% 3 2519 1656 2519 1752 5 0
% $A(\{1,3\})$
\put(25.1900,-17.5200){\makebox(0,0){$A(\{1,3\})$}}%
% STR 2 0 3 0
% 3 2519 2236 2519 2333 5 0
% $A(\{1,4\})$
\put(25.1900,-23.3300){\makebox(0,0){$A(\{1,4\})$}}%
% STR 2 0 3 0
% 3 1551 2241 1551 2338 5 0
% $A(\{2,4\})$
\put(15.5100,-23.3800){\makebox(0,0){$A(\{2,4\})$}}%
% STR 2 0 3 0
% 3 583 2246 583 2343 5 0
% $A(\{3,4\})$
\put(5.8300,-23.4300){\makebox(0,0){$A(\{3,4\})$}}%
% LINE 2 0 3 0
% 2 2519 1268 2519 1656
% 
\special{pn 8}%
\special{pa 2520 1268}%
\special{pa 2520 1656}%
\special{fp}%
% LINE 2 0 3 0
% 2 2519 1849 2519 2236
% 
\special{pn 8}%
\special{pa 2520 1850}%
\special{pa 2520 2236}%
\special{fp}%
% LINE 2 0 3 0
% 2 1551 1854 1551 2241
% 
\special{pn 8}%
\special{pa 1552 1854}%
\special{pa 1552 2242}%
\special{fp}%
% STR 2 0 3 0
% 3 3680 1080 3680 1176 5 0
% $B(\{1,2\})$
\put(36.8000,-11.7600){\makebox(0,0){$B(\{1,2\})$}}%
% STR 2 0 3 0
% 3 3680 1660 3680 1757 5 0
% $B(\{1,3\})$
\put(36.8000,-17.5700){\makebox(0,0){$B(\{1,3\})$}}%
% STR 2 0 3 0
% 3 3680 2241 3680 2338 5 0
% $B(\{1,4\})$
\put(36.8000,-23.3800){\makebox(0,0){$B(\{1,4\})$}}%
% LINE 2 0 3 0
% 2 3680 1273 3680 1660
% 
\special{pn 8}%
\special{pa 3680 1274}%
\special{pa 3680 1660}%
\special{fp}%
% LINE 2 0 3 0
% 2 3680 1854 3680 2241
% 
\special{pn 8}%
\special{pa 3680 1854}%
\special{pa 3680 2242}%
\special{fp}%
% STR 2 0 3 0
% 3 4629 1660 4629 1757 5 0
% $B(\{2,3\})$
\put(46.2900,-17.5700){\makebox(0,0){$B(\{2,3\})$}}%
% STR 2 0 3 0
% 3 4629 2241 4629 2338 5 0
% $B(\{2,4\})$
\put(46.2900,-23.3800){\makebox(0,0){$B(\{2,4\})$}}%
% LINE 2 0 3 0
% 2 4629 1854 4629 2241
% 
\special{pn 8}%
\special{pa 4630 1854}%
\special{pa 4630 2242}%
\special{fp}%
% STR 2 0 3 0
% 3 5577 2246 5577 2343 5 0
% $B(\{3,4\})$
\put(55.7700,-23.4300){\makebox(0,0){$B(\{3,4\})$}}%
% LINE 2 0 3 0
% 2 2180 1752 1889 1752
% 
\special{pn 8}%
\special{pa 2180 1752}%
\special{pa 1890 1752}%
\special{fp}%
% LINE 2 0 3 0
% 2 2180 2333 1889 2333
% 
\special{pn 8}%
\special{pa 2180 2334}%
\special{pa 1890 2334}%
\special{fp}%
% LINE 2 0 3 0
% 2 1212 2333 921 2333
% 
\special{pn 8}%
\special{pa 1212 2334}%
\special{pa 922 2334}%
\special{fp}%
% LINE 2 0 3 0
% 2 4309 2333 4019 2333
% 
\special{pn 8}%
\special{pa 4310 2334}%
\special{pa 4020 2334}%
\special{fp}%
% LINE 2 0 3 0
% 2 4309 1752 4019 1752
% 
\special{pn 8}%
\special{pa 4310 1752}%
\special{pa 4020 1752}%
\special{fp}%
% LINE 2 0 3 0
% 2 5260 2330 4970 2330
% 
\special{pn 8}%
\special{pa 5260 2330}%
\special{pa 4970 2330}%
\special{fp}%
% STR 2 0 3 0
% 3 3200 503 3200 600 5 0
% Figure 3
\put(32.0000,-6.0000){\makebox(0,0){Figure 3}}%
% STR 2 0 3 0
% 3 3099 978 3099 1075 5 0
% 45
\put(30.9900,-10.7500){\makebox(0,0){45}}%
% STR 2 0 3 0
% 3 3099 1559 3099 1656 5 0
% 5
\put(30.9900,-16.5600){\makebox(0,0){5}}%
% STR 2 0 3 0
% 3 3099 2140 3099 2236 5 0
% 3
\put(30.9900,-22.3600){\makebox(0,0){3}}%
% STR 2 0 3 0
% 3 4164 2140 4164 2236 5 0
% 2
\put(41.6400,-22.3600){\makebox(0,0){2}}%
% STR 2 0 3 0
% 3 4164 1559 4164 1656 5 0
% 2
\put(41.6400,-16.5600){\makebox(0,0){2}}%
% STR 2 0 3 0
% 3 5132 2140 5132 2236 5 0
% 3
\put(51.3200,-22.3600){\makebox(0,0){3}}%
% STR 2 0 3 0
% 3 2035 2140 2035 2236 5 0
% 2
\put(20.3500,-22.3600){\makebox(0,0){2}}%
% STR 2 0 3 0
% 3 2035 1559 2035 1656 5 0
% 2
\put(20.3500,-16.5600){\makebox(0,0){2}}%
% STR 2 0 3 0
% 3 1067 2140 1067 2236 5 0
% 3
\put(10.6700,-22.3600){\makebox(0,0){3}}%
% STR 2 0 3 0
% 3 3758 1375 3758 1472 5 0
% 3
\put(37.5800,-14.7200){\makebox(0,0){3}}%
% STR 2 0 3 0
% 3 3758 1956 3758 2052 5 0
% 4
\put(37.5800,-20.5200){\makebox(0,0){4}}%
% STR 2 0 3 0
% 3 4726 1956 4726 2052 5 0
% 4
\put(47.2600,-20.5200){\makebox(0,0){4}}%
% STR 2 0 3 0
% 3 2402 1956 2402 2052 5 0
% 4
\put(24.0200,-20.5200){\makebox(0,0){4}}%
% STR 2 0 3 0
% 3 1434 1956 1434 2052 5 0
% 4
\put(14.3400,-20.5200){\makebox(0,0){4}}%
% STR 2 0 3 0
% 3 2402 1375 2402 1472 5 0
% 3
\put(24.0200,-14.7200){\makebox(0,0){3}}%
% STR 2 0 3 0
% 3 3110 743 3110 840 5 0
% 15
\put(31.1000,-8.4000){\makebox(0,0){15}}%
% STR 2 0 3 0
% 3 3099 2517 3099 2614 5 0
% 1
\put(30.9900,-26.1400){\makebox(0,0){1}}%
% STR 2 0 3 0
% 3 3100 2743 3100 2840 5 0
% 12
\put(31.0000,-28.4000){\makebox(0,0){12}}%
% LINE 2 0 3 0
% 6 2906 2333 3293 2333 2906 1752 3293 1752 2906 1172 3293 1172
% 
\special{pn 8}%
\special{pa 2906 2334}%
\special{pa 3294 2334}%
\special{fp}%
\special{pa 2906 1752}%
\special{pa 3294 1752}%
\special{fp}%
\special{pa 2906 1172}%
\special{pa 3294 1172}%
\special{fp}%
% ELLIPSE 2 0 3 0
% 4 3099 2333 1551 2720 1551 2430 4648 2430
% 
\special{pn 8}%
\special{ar 3100 2334 1548 388  0.2454343 2.8960061}%
% ELLIPSE 2 0 3 0
% 4 3099 2333 583 2914 583 2430 5616 2430
% 
\special{pn 8}%
\special{ar 3100 2334 2516 582  0.1653420 2.9761862}%
% ELLIPSE 2 0 3 0
% 4 3110 1750 4630 930 4630 1640 1560 1640
% 
\special{pn 8}%
\special{ar 3110 1750 1520 820  3.2724534 6.1497720}%
\end{picture}%
\end{center}
\end{example}
\medskip

\noindent
The numbers on edges are colors of edges.
For example, there are two edges between $A(\{1,2\})$ and $B(\{1,2\})$
such that one edge has color 4 and the other edge has color 5.
In Figure 3, we omit edges between $A$ and $C$ (and $B$ and $D$),
but they are edges between $A(S)$ and $C(S)$ whose colors are the colors which do not appear in $A(S)$.
For example, there are two edges between $A(\{1,2\})$ and $C(\{1,2\})$
whose color is $1$ or $2$.
%such that one edge has color $1$ and the other edge has color $2$.

\section{Proof of Theorem \ref{1.2}}

In this section, we prove the sufficiency of Theorem \ref{1.2}.

\subsection{Cancellations of dipoles}
\

Let $\Lambda=(G,\gamma)$ with $G=(V,E,\phi)$ be an admissible $d$-colored multi-graph.
Let $x,y \in V$ be vertices of $G$.
We define a new admissible $d$-colored multi-graph $\Lambda'=\del_{\{x,y\}} \Lambda =(G',\gamma')$ with $G'=(V',E',\phi')$
as follows:
Let $C=\{\gamma(e): e \in E,\ \phi(e)=\{x,y\}\}$.
Thus $C$ is the set of colors of edges between $x$ and $y$.
Then, for each $i \in [d] \setminus C$,
there is the unique pair $(a_i,b_i)$ of vertices in $G$ such that
there are edges $e$ and $e'$ in $E$ of color $i$
with $\phi(e)=\{a_i,x\}$ and with $\phi(e')=\{y,b_i\}$.
Then we define the graph $G'=(V',E',\phi')$ by
\begin{eqnarray*}
V' &=& V \setminus \{x,y\}\\
E' &=& \{ e \in E: \phi(e)\cap \{x,y\} = \emptyset\} \cup \{f_i: i \in [d] \setminus C\},\\
\phi'(e)&=&\left\{
\begin{array}{ll}
\phi(e),& \mbox{ if $e \in E$,}\\
\{a_i,b_i\},& \mbox{ if $e = f_i$ for some $i \in [d]\setminus C$.}
\end{array}
\right.
\end{eqnarray*}
Also, we define the coloring $\gamma'$ of $G'$ by
\begin{eqnarray*}
\gamma'(e)=\left\{
\begin{array}{ll}
\gamma(e),& \mbox{ if $e \in E$,}\\
i, & \mbox{ if $e = f_i$ for some $i \in [d] \setminus C$.}
\end{array}
\right.
\end{eqnarray*}

Thus $\Lambda'$ is the graph obtained from $\Lambda$ by removing the vertices $x$ and $y$
and by adding, for each color $i \in [d] \setminus C$,
a new edge $f_i$ of color $i$ between the vertices $a_i$ and $b_i$
(see Figure 4).
By the construction, it is easy to see that $\Lambda'$ is an admissible $d$-colored multi-graph.
We call the operation $\Lambda \to \Lambda'$ a \textit{cancelling (of $x$ and $y$)}.

\begin{center}
%WinTpicVersion3.08
\unitlength 0.1in
\begin{picture}( 53.1500, 17.6500)(  7.7500,-18.9000)
% LINE 2 0 3 0
% 10 1000 600 1600 1200 1600 1200 1000 1800 1600 1200 2400 1200 2400 1200 3000 600 2400 1200 3000 1800
% 
\special{pn 8}%
\special{pa 1000 600}%
\special{pa 1600 1200}%
\special{fp}%
\special{pa 1600 1200}%
\special{pa 1000 1800}%
\special{fp}%
\special{pa 1600 1200}%
\special{pa 2400 1200}%
\special{fp}%
\special{pa 2400 1200}%
\special{pa 3000 600}%
\special{fp}%
\special{pa 2400 1200}%
\special{pa 3000 1800}%
\special{fp}%
% CIRCLE 2 0 2 0
% 4 3000 600 3000 690 3000 690 3000 690
% 
\special{pn 8}%
\special{sh 0}%
\special{ar 3000 600 90 90  0.0000000 6.2831853}%
% CIRCLE 2 0 2 0
% 4 3000 1800 3000 1890 3000 1890 3000 1890
% 
\special{pn 8}%
\special{sh 0}%
\special{ar 3000 1800 90 90  0.0000000 6.2831853}%
% CIRCLE 2 0 2 0
% 4 2400 1200 2400 1290 2400 1290 2400 1290
% 
\special{pn 8}%
\special{sh 0}%
\special{ar 2400 1200 90 90  0.0000000 6.2831853}%
% CIRCLE 2 0 2 0
% 4 1600 1200 1600 1290 1600 1290 1600 1290
% 
\special{pn 8}%
\special{sh 0}%
\special{ar 1600 1200 90 90  0.0000000 6.2831853}%
% CIRCLE 2 0 2 0
% 4 1000 600 1000 690 1000 690 1000 690
% 
\special{pn 8}%
\special{sh 0}%
\special{ar 1000 600 90 90  0.0000000 6.2831853}%
% CIRCLE 2 0 2 0
% 4 1000 1800 1000 1890 1000 1890 1000 1890
% 
\special{pn 8}%
\special{sh 0}%
\special{ar 1000 1800 90 90  0.0000000 6.2831853}%
% STR 2 0 3 0
% 3 1600 1100 1600 1200 5 0
% $x$
\put(16.0000,-12.0000){\makebox(0,0){$x$}}%
% STR 2 0 3 0
% 3 2400 1100 2400 1200 5 0
% $y$
\put(24.0000,-12.0000){\makebox(0,0){$y$}}%
% ELLIPSE 2 0 3 0
% 4 5000 600 6000 1000 4000 600 6000 600
% 
\special{pn 8}%
\special{ar 5000 600 1000 400  6.2831853 6.2831853}%
\special{ar 5000 600 1000 400  0.0000000 3.1415927}%
% ELLIPSE 2 0 3 0
% 4 5000 1800 6000 1400 6000 1800 4000 1800
% 
\special{pn 8}%
\special{ar 5000 1800 1000 400  3.1415927 6.2831853}%
% CIRCLE 2 0 2 0
% 4 4000 600 4000 690 4000 690 4000 690
% 
\special{pn 8}%
\special{sh 0}%
\special{ar 4000 600 90 90  0.0000000 6.2831853}%
% CIRCLE 2 0 2 0
% 4 4000 1800 4000 1890 4000 1890 4000 1890
% 
\special{pn 8}%
\special{sh 0}%
\special{ar 4000 1800 90 90  0.0000000 6.2831853}%
% CIRCLE 2 0 2 0
% 4 6000 1800 6000 1890 6000 1890 6000 1890
% 
\special{pn 8}%
\special{sh 0}%
\special{ar 6000 1800 90 90  0.0000000 6.2831853}%
% CIRCLE 2 0 2 0
% 4 6000 600 6000 690 6000 690 6000 690
% 
\special{pn 8}%
\special{sh 0}%
\special{ar 6000 600 90 90  0.0000000 6.2831853}%
% STR 2 0 3 0
% 3 2000 1000 2000 1100 5 0
% 3
\put(20.0000,-11.0000){\makebox(0,0){3}}%
% STR 2 0 3 0
% 3 1340 650 1340 750 5 0
% 1
\put(13.4000,-7.5000){\makebox(0,0){1}}%
% STR 2 0 3 0
% 3 1340 1540 1340 1640 5 0
% 2
\put(13.4000,-16.4000){\makebox(0,0){2}}%
% STR 2 0 3 0
% 3 2700 1540 2700 1640 5 0
% 2
\put(27.0000,-16.4000){\makebox(0,0){2}}%
% STR 2 0 3 0
% 3 2700 640 2700 740 5 0
% 1
\put(27.0000,-7.4000){\makebox(0,0){1}}%
% STR 2 0 3 0
% 3 5000 770 5000 870 5 0
% 1
\put(50.0000,-8.7000){\makebox(0,0){1}}%
% STR 2 0 3 0
% 3 5000 1510 5000 1610 5 0
% 2
\put(50.0000,-16.1000){\makebox(0,0){2}}%
% STR 2 0 3 0
% 3 3500 110 3500 210 5 0
% Figure 4
\put(35.0000,-2.1000){\makebox(0,0){Figure 4}}%
% STR 2 0 3 0
% 3 3500 1100 3500 1200 5 0
% $\Rightarrow$
\put(35.0000,-12.0000){\makebox(0,0){$\Rightarrow$}}%
% STR 2 0 3 0
% 3 1000 500 1000 600 5 0
% $a_1$
\put(10.0000,-6.0000){\makebox(0,0){$a_1$}}%
% STR 2 0 3 0
% 3 1000 1700 1000 1800 5 0
% $a_2$
\put(10.0000,-18.0000){\makebox(0,0){$a_2$}}%
% STR 2 0 3 0
% 3 3000 1700 3000 1800 5 0
% $b_2$
\put(30.0000,-18.0000){\makebox(0,0){$b_2$}}%
% STR 2 0 3 0
% 3 3000 500 3000 600 5 0
% $b_1$
\put(30.0000,-6.0000){\makebox(0,0){$b_1$}}%
% STR 2 0 3 0
% 3 4000 500 4000 600 5 0
% $a_1$
\put(40.0000,-6.0000){\makebox(0,0){$a_1$}}%
% STR 2 0 3 0
% 3 4000 1700 4000 1800 5 0
% $a_2$
\put(40.0000,-18.0000){\makebox(0,0){$a_2$}}%
% STR 2 0 3 0
% 3 6000 1700 6000 1800 5 0
% $b_2$
\put(60.0000,-18.0000){\makebox(0,0){$b_2$}}%
% STR 2 0 3 0
% 3 6000 500 6000 600 5 0
% $b_1$
\put(60.0000,-6.0000){\makebox(0,0){$b_1$}}%
\end{picture}%
\end{center}
\medskip

We say that two vertices $u$ and $v$ of a multi-graph $G$ are \textit{connected
on $G$} if there exists a sequence of edges, called a \textit{path}, $e_1,e_2,\dots,e_l$ of $G$
such that $u \in \phi(e_1)$, $v \in \phi(e_l)$ and $\phi(e_i)\cap \phi(e_{i+1}) \ne \emptyset$
for $i=1,2,\dots,l-1$.
If two vertices $u$ and $v$ are not connected on $G$,
then we say that they are \textit{disconnected on $G$}.

Let $C=\{\gamma(e): e \in E,\ \phi(e)=\{x,y\}\}$.
The cancelling $\Lambda \to \Lambda'=\del _{\{x,y\}} \Lambda$ is said to be a \textit{cancelling of a dipole (of type $C$)}
if $C \ne \emptyset$ and the vertices $x$ and $y$ are disconnected on $\Lambda_{[d] \setminus C}$.
The following result is known in crystallization theory (\cite[Lemma 1]{FG}).

\begin{lemma}[Ferri-Gagliardi]
\label{6.1}
Let $\Lambda$ be an admissible $d$-colored multi-graph such that
(the barycentric subdivision of)
the simplicial cell complex $\Gamma(P_\Lambda)$ is a PL-manifold.
If $\Lambda \to \Lambda'$ is a cancellation of a dipole
then
$\Gamma(P_{\Lambda'})$ is homeomorphic to $\Gamma(P_\Lambda)$.
\end{lemma}

\subsection{Construction}\

Recall that by Corollary \ref{3.2},
to prove Theorem \ref{1.2},
it is enough to construct a simplicial cell decomposition $P$
of $S^n \times S^m$
with $h''(P)=(1,0,\dots,0,1)$.

\begin{lemma}
\label{6.2}
Let $\Lambda$ be an admissible $(n+m+1)$-colored multi-graph such that $P_\Lambda$
is a simplicial cell decomposition of $S^n \times S^m$.
If the number of vertices of $\Lambda$ is equal to $2+ 2{n+m \choose n}$
then $h''(P_\Lambda)=(1,0,\dots,0,1)$.
\end{lemma}

\begin{proof}
Observe $\beta_i(P_\Lambda)=0$ if $i \ne n,m,n+m$,
$\beta_n(P_\Lambda)=\beta_m(P_\Lambda)=1$ if $n \ne m$ and $\beta_n(P_\Lambda)=2$ if $n=m$.
In the proof of Corollary \ref{2.3},
we show
$$f_{n+m+1}(P_\Lambda)=\sum_{i=0}^{n+m+1} h_i(P_\Lambda)
=\sum_{i=0}^{n+m+1} h_i''(P_\Lambda) + {n+m \choose n} + {n+m \choose m}.$$
Since the number of vertices of $\Lambda$ is equal to $f_{n+m+1}(P_\Lambda)$,
$\sum_{i=0}^{n+m+1} h_i''(P_\Lambda) =2$.
Then the statement follows since $h''(P_\Lambda)$ is non-negative.
\end{proof}

By the above lemma,
to prove Theorem \ref{1.2},
what we must prove is
the existence of a crystallization of $S^n \times S^m$
with $2+2{n+m \choose n}$ vertices.
Unfortunately,
the graph $\Lambda(n,m)$ given in the previous section has $4 {n+m \choose n}$ vertices.
We make a desired crystallization
by repeating cancellations of dipoles to $\Lambda(n,m)$.

From now on we fix positive integers $n$ and $m$.
For integers $i,j$, we write $[i,j]=\{i,i+1,\dots,j\}$ where $[i,j]=\emptyset$ if $j<i$.
For $j=1,2,\dots,n$, let
$$X_j=\{ S \subset [j+1,n+m]: \# S  = n+1-j\}$$
and
$$X= \bigcup_{j=1}^n X_j.$$

\begin{remark}
\label{6.3}
There is a natural bijection $\Phi:X \to \{S \subset [n+m]:  \# S= n\} \setminus \{[n]\}$
defined by $\Phi(S)=[n-\# S] \cup S$.
In particular $\# X={n+m \choose n} -1$.
\end{remark}

In the rest of this section,
for a set $\{i_1,\dots,i_k\}$ of integers,
we always assume $i_1 < \cdots <i_k$.

\begin{definition}
\label{6.4}
Recall that the vertices of $\Lambda(n,m)$ are denoted by $A(S),B(S), C(S)$ and $D(S)$
where $S \subset [n+m]$ and $ \# S=n$.
For $S=\{i_1,i_2,\dots,i_{n+1-j}\} \in X_j$
we define the pair $\D_j(S)$ of vertices of $\Lambda(n,m)$ as follows:
Let $S'=S \setminus \{i_1\}$.
If $j$ is odd then
\begin{eqnarray*}
\ \ \D_j(S)=\left\{
\begin{array}{ll}
\big\{ A \big([j-1]\cup S \big), A\big([j] \cup S' \big) \big\}, & \mbox{ if } i_1 =j+1,\smallskip\\
\big\{ A\big([j-1]\cup S\big), B \big([i_1-j-1,i_1-2] \cup S' \big) \big\}, & \mbox{ if }i_1 >j+1,
\end{array}
\right.
\end{eqnarray*}
and if $j$ is even then
\begin{eqnarray*}
\D_j(S)=\left\{ B\big([i_1-j,i_1-2] \cup S\big), B\big([i_1-j,i_1-1] \cup S'\big)\right\}.
\end{eqnarray*}
Since $j=n+1-\# S$,
we simply write $\D_j(S)=\D(S)$.
\end{definition}

\begin{lemma}
\label{6.5}
\
\begin{itemize}
\item[(1)] Suppose that $n$ is even. Then
\begin{itemize}
\item[(a)] for any $F \subset [n+m]$ with $\# F=n$ and $F \ne [n]$,
there is the unique $S \in X$ such that $A(F) \in \D(S)$.
\item[(b)] for any $F \subset [n+m]$ with $\# F=n$ and $F \ne [m+1,m+n]$,
there is the unique $S \in X$ such that $B(F) \in \D(S)$.
\end{itemize}
\item[(2)] Suppose that $n$ is odd. Then
\begin{itemize}
\item[(a)] for any $F \subset [n+m]$ with $\# F=n$,
there is the unique $S \in X$ such that $A(F) \in \D(S)$.
\item[(b)] for any $F \subset [n+m]$ with $\# F=n$, $F \ne [m,m+n-1]$ and $F \ne [m+1,m+n]$,
there is the unique $S \in X$ such that $B(F) \in \D(S)$.
\end{itemize}
\end{itemize}
\end{lemma}

\begin{proof}
The uniqueness follows from the existence.
Indeed, the number of vertices $A(F)$ and $B(F)$ appearing in (a) and (b)
is $2 {n+m \choose n} -2$.
On the other hand, since $\# X={n+m \choose n} -1$,
the number of vertices which appears in $\D(S)$ for some $S \in X$
is at most $2 {n+m \choose n}-2$.
Thus if it is $2{n+m \choose n} -2$ then each vertex cannot appear twice.
We will prove the existence.
Let $F=\{i_1,i_2,\dots,i_n\} \subset [n+m]$.

We first consider $A(F)$.
Let $k$ be the smallest positive integer which is not in $F$.
Thus $F \supset [k-1]$ and $k \not \in F$.
Let $F=[k-1] \cup F'$ where $\min F' >k$.
If $k$ is odd and $F' \ne \emptyset$ then $A(F)$ is the first vertex of $\D_k(F')$.
If $k$ is even then $A(F)$ is the second vertex of $\D_{k-1}(\{k\} \cup F')$.
These prove (1)-(a) and (2)-(a).

Next, we consider $B(F)$.
Let $l$ be the smallest integer such that $i_{l+1} \ne i_l +1$.
Thus $F \supset \{i_1,i_1+1,\dots,i_1+l-1\}=\{i_1,\dots,i_l\}$
and $i_1+l \not \in F$.
Let $F= [i_1,i_1+l-1]\cup F'$ with $\min F >i_1+l$.

Suppose $F' \ne \emptyset$.
Then $l <n$.
If $l$ is even then $B(F)$ is the second vertex of $\D_l(\{i_1+l\} \cup F')$.
If $l$ is odd and $i_{l+1} >i_1 +l+1$ then $B(F)$ is the second vertex of $\D_l(\{i_1+l+1\}  \cup F')$.
If $l$ is odd and $i_{l+1}=i_1+l+1$ then $B(F)$ is the first vertex of $\D_{l+1}(F')$.

Suppose $F' = \emptyset$, that is, $F=[i_1,i_1+n-1]$ with $1 \leq i_1 \leq m+1$.
If $n$ is even and $i_1 \ne m+1$ then $B(F)$ is the second vertex of $\D_n(\{i_1+n\})$.
If $n$ is odd and $i_1 <m$ then $B(F)$ is the second vertex of $\D_n(\{i_1+n+1\})$.
\end{proof}

We define the total order $\succ$ on $\{\D(S): S \in X\}$ by
$\D_j(S)\succ \D_{j'}(S')$ if (i) $j<j'$ or (ii) $j=j'$ and $S >_\rlex S'$,
where $>_\rlex$ is the reverse lexicographic order.
Thus $S> _\rlex S'$ if the largest integer in the symmetric difference $(S\setminus S') \cup (S' \setminus S)$ is contained in $S'$.
From the proof of Lemma \ref{6.5},
we obtain the next corollary.

\begin{corollary}
\label{6.6.6}
%Let $F \subset [n+m]$ with $\# F=n$.
%\begin{itemize}
%\item[(i)] Let $F=[k-1] \cup F'$ with $\min F' >k$.
%If $A(F)$ appears in $\D_l(S)$ then $l<k$ or $S=F'$.
%\item[(ii)] 
Let $F=[i_1,i_1+k-1] \cup F' \subset [n+m]$ with $\min F' >i_1+k$ and with $\# F=n$.
If $B(F)$ appears in $\D_l(T)$ then 
$l \geq k$ and
$\D_l(T) \preceq \D_k(\{i_1+k\} \cup F')$.
%\end{itemize}
\end{corollary}

Let
$$\big\{\D(S) : S \in X\big\}
=\left\{\D_1 \succ \D_2 \succ \cdots \succ \D_{{n+m \choose n}-1}\right\}.$$
We define the admissible $(n+m+1)$-colored multi-graph $\Lambda^{(k)}$ recursively by
$\Lambda^{(1)}=\Lambda(n,m)$ and
$$\Lambda^{(k+1)}=\del_{\D_{k}} \Lambda^{(k)}$$
for $k=1,2,\dots,{n+ m \choose n} -1$
(these graphs are well defined by Lemma \ref{6.5}).
If $\D_k=\D_j(S)$, we write
$$\Lambda(S)=\Lambda_j(S)=\Lambda^{(k)}
\mbox{ and } \Lambda'(S)=\Lambda^{(k+1)}= \del_{\D(S)} \Lambda^{(k)}.$$
Clearly
the number of vertices of $\Lambda^{(k)}$ is $4 {n+m \choose n} - 2(k-1)$.
Then by Lemma \ref{6.1} the next statement completes the proof of Theorem \ref{1.2}.

\begin{lemma}
\label{6.7}
For $k=1,2,\dots,{n+m \choose n}-1$,
the cancelling $\Lambda^{(k)} \to \Lambda^{(k+1)}$
is a cancelling of a dipole.
\end{lemma}

We prove the above lemma in subsections 6.3 and 6.4
in a series of lemmas.

\begin{example}
Suppose $n=m=2$.
Then
\begin{eqnarray*}
\D_1 &=&\D_1(\{2,3\})=\{ A(\{2,3\}),A(\{1,3\})\},\\
\D_2 &=& \D_1(\{2,4\})=\{ A(\{2,4\}),A(\{1,4\})\},\\
\D_3 &=& \D_1(\{3,4\})=\{ A(\{3,4\}),B(\{1,4\})\},\\
\D_4 &=& \D_2(\{3\})=\{ B(\{1,3\}),B(\{1,2\})\},\\
\D_5 &=& \D_2(\{4\})=\{ B(\{2,4\}),B(\{2,3\})\}.
\end{eqnarray*}
\end{example}

\subsection{Proof of Lemma \ref{6.7}: disconnectivity of $\Lambda(S)$}\

Let $\Lambda^{(k)}=\Lambda_j(S)$
and let
$$\color (S)=\{l \in S : l-1 \not \in S\}.$$
The next lemma gives a part of a proof of Lemma \ref{6.7}.

\begin{lemma}
\label{6.8}
With the same notation as above,
two vertices in $\D(S)$ are disconnected on $(\Lambda(S))_{[n+m+1] \setminus \color(S)}$.
\end{lemma}

We need the following technical but obvious lemma.

\begin{lemma}
\label{6.9}
Let $\Lambda$ be an admissible $d$-colored multi-graph on the vertex set $V$ and $T \subset [d]$.
Let $X \cup (V \setminus X)$ be a partition of $V$ such that,
for all $x \in X$ and $y \in V \setminus X$,
$x$ and $y$ are disconnected on $\Lambda_T$.
If $u$ and $v$ are vertices in $X$,
then, for all $x \in X \setminus \{u,v\}$ and $y \in V \setminus X$,
$x$ and $y$ are disconnected on $(\del_{\{u,v\}} \Lambda)_T$.
\end{lemma}

\begin{proof}[Proof of Lemma \ref{6.8}]
Let
$$S=\{i_1,i_2,\dots,i_{n+1-j}\}.$$
For $\{p_1,p_2,\dots,p_{n+1-j}\} \subset [n+m]$,
we write $\{p_1,p_2,\dots,p_{n+1-j}\} \sqsupseteq  S$
if $p_l \geq i_l$ for all $l$.
Let
\begin{eqnarray*}
X&=&\big\{A\big(\{p_1,\dots,p_n\}\big): \{p_j,\dots,p_n\} \sqsupseteq  S\big\}\\ 
&&\bigcup\big\{B\big(\{p_1,\dots,p_n\}\big): \{p_j,\dots,p_n\} \sqsupseteq  S\big\}\\
&&\bigcup\big\{C\big(\{p_1,\dots,p_n\}\big): \{p_j,\dots,p_n\} \sqsupseteq  S\big\}\\
&&\bigcup\big\{D\big(\{p_1,\dots,p_n\}\big): \{p_j,\dots,p_n\} \sqsupseteq  S\big\}.
\end{eqnarray*}
Let $V$ be the set of vertices of $\Lambda(n,m)$.
We claim that the partition $X \cup (V \setminus X)$
and the set of colors $T=[n+m+1] \setminus \color(S)$
satisfy the assumption of Lemma \ref{6.9} for $\Lambda(n,m)$.

We use the description (E1), (E2) and (E3) of edges of $\Lambda(n,m)$.
By the description,
if $e_1,\dots,e_l$ is a path on $\Lambda(n,m)$ from $x \in X$ to $y \in V \setminus X$,
then there is an edge $e_q$ whose vertices are of the form
$$x'=\diamondsuit \big(\{p_1,\dots,p_{j-1},i_1,i_2,\dots,i_{n+1-j}\}\big)$$
and
$$y'=\diamondsuit \big(\{p_1,\dots,p_{j-1},i_1,\dots,i_{\ell-1},i_\ell-1,i_{\ell+1},\dots,,i_{n+1-j}\}\big)$$
with $i_\ell-1 \not \in S$,
where $\diamondsuit$ is $A,B,C$ or $D$.
Also such an edge $e_q$ has color $i_\ell$ by (E3).
Since $i_l \in S$, we have $e_q \not \in \Lambda(n,m)_T$.
Thus $e_1,\dots,e_l$ is not a path on $\Lambda(n,m)_T$.
This fact shows that $x \in X$ and $y \in V\setminus S$ are disconnected on $\Lambda(n,m)_T$.

Now by lemma \ref{6.9}
what we must prove is that, for any integer $k' <k$,
$\D_{k'}$ is contained in either $X$ or $V \setminus X$.
Let $\D_{k'}=\D_{j'}(T)=\{\square(\{p_1,\dots,p_n\}),\square'(\{q_1,\dots,q_n\}) \}$
where $\square$ and $\square'$ are either $A$ or $B$.
Since $\D_{k'} \succ \D_k$, we have $j'<j$ or $j'=j$ and $T >_\rlex S$.
If $j' <j$ then $\{p_j,\dots,p_n\}=\{q_j,\dots,q_n\}$ by Definition \ref{6.4},
which guarantees $\D_{k'} \subset X$ or $\D_{k'} \subset V \setminus X$.
Suppose $j'=j$.
Then $T>_\rlex S$.
By Definition \ref{6.4},
$\{p_j,\dots,p_n\} \geq_\rlex T$ and $\{q_j,\dots,q_n\} \geq_\rlex T$.
Hence we have $\D_{k'} \subset V \setminus X$, as desired.
\end{proof}

\subsection{Proof of Lemma \ref{6.7}: existence of edges with desired colors}\

We say that two vertices $u$ and $v$ in $\Lambda^{(k)}$ are
\textit{directly connected on $\Lambda^{(k)}$ by colors $H \subset [n+m+1]$} if, for each $i \in H$,
there is an edge $e$ of $\Lambda^{(k)}$ whose vertices are $u$ and $v$ and whose color is $i$.
The next lemma and Lemma \ref{6.8}
prove Lemma \ref{6.7}.

\begin{lemma}
\label{6.10}
For every $S \in X$, the vertices in $\D(S)$ are directly connected on $\Lambda(S)$ by colors $\color (S)$.
\end{lemma}

We need two technical lemmas.

\begin{lemma}
\label{6.11}
Let $S=\{i_1,\dots,i_{n+1-j}\} \in X_j$ and $S'=S \setminus \{i_1\}$.
\begin{itemize}
\item[(i)] Suppose $j$ is odd. Then $B([i_1-j,i_1-1] \cup S')$
and $B([i_1-j+1,i_1] \cup S')$ are vertices of $\Lambda(S)$.
Moreover, if $i_1+1 \not \in S$ and $i_1+1 \leq n+m$, then $A([j-1]\cup \{i_1+1\} \cup S')$ is a vertex of $\Lambda(S)$.
\item[(ii)] Suppose $j$ is even.
Then $A([j] \cup S')$ and $B([i_1-j+1,i_1] \cup S')$ are vertices of $\Lambda(S)$.
\end{itemize}
\end{lemma}

\begin{proof}
By Lemma \ref{6.5},
to prove that $A(F)$ (or $B(F)$) is a vertex of $\Lambda(S)$,
what we must prove is that it appears in some $\D(T)$ with $\D(T) \preceq \D(S)$
or it does not appear in any $\D(T)$.

(i) If $B([i_1-j,i_1-1] \cup S')$ or $B([i_1-j+1,i_1] \cup S')$
appears in some $\D(T)$, then Corollary \ref{6.6.6} says $\D(T) \preceq \D_j(\{ i_1\} \cup S')=\D_j(S)$.
If $i_1+1 \not \in S$ then $A([j-1] \cup \{i_1+1\} \cup S')$
appears in $\D_j(\{i_1+1\} \cup S') \prec \D_j(S)$ by Definition \ref{6.4}.

(ii)
If $A([j] \cup S')$ appears in some $\D_l(T)$,
then by Definition \ref{6.4}
$A([j] \cup S') \in \D_{j+1}(S') \prec D_j(S)$.
Also, if $B([i_1+j+1, i_1] \cup S')$ appears in some $\D(T)$
then Corollary \ref{6.6.6} says
$\D(T) \preceq \D(\{i_1+1\} \cup S') \prec \D_j(S)$.
\end{proof}

\begin{lemma}
\label{6.12}
Let $S=\{i_1,\dots,i_{n+1-j}\} \in X_j$ and $S'=S \setminus \{i_1\}$.
\begin{itemize}
\item[(i)] If $j$ is odd then $B([i_1-j,i_1-1] \cup S')$ and
$B([i_1-j+1,i_1]\cup S')$ are directly connected on $\Lambda'(S)$ by colors
$$H=\big\{ r \in [i_1+2,n+m+1]: \{r-1,r \} \cap S'=\emptyset\big\}.$$
\item[(ii)] If $j$ is odd and $i_1+1 \not \in S$, where $i_1+1 \leq n+m$,
then $A([j-1] \cup\{i_1+1\} \cup S')$ and
$B([i_1-j,i_1-1]\cup S')$ are directly connected on $\Lambda'(S)$ by color $i_1+1$.
\item[(iii)] If $j$ is even then $A([j] \cup S')$ and $B([i_1-j+1,i_1] \cup S')$
are directly connected on $\Lambda'(S)$ by colors
$$H'=\big\{ r \in [i_2+2,n+m+1]: \{r-1,r \} \cap S'=\emptyset\big\}.$$
\end{itemize}
\end{lemma}

\begin{proof}
We prove the statement by induction on the total order $\succ$ on $\{\D(T): T \in X\}$.
Note that all vertices appearing in the statements are vertices of $\Lambda(S)$
by Lemma \ref{6.11}.

We often use the following fact:
if two vertices are directly connected on $\Lambda^{(k)}$
by colors $C$ and if they are still vertices of $\Lambda^{(l)}$ with $l>k$
then they are  
directly connected on $\Lambda^{(l)}$ by colors $C$.

\textit{Case 1.}
Suppose $j$ is odd and $i_1=j+1$.
Then
$$\D(S)=\big\{A \big([j-1] \cup S \big), A \big([j] \cup S' \big) \big\}.$$
By the description (E1) of edges in $\Lambda(n,m)$,
\begin{itemize}
\item[(6.1)]
$A([j] \cup S')$ and $B([i_1-j,i_1-1] \cup S')=B([j] \cup S')$
are directly connected on $\Lambda(n,m)$ by colors $H$ (and $j+2$ if $j+2 \not \in S$).
\end{itemize}
By applying the induction hypothesis to $\tilde S= \{j \} \cup S \in X_{j-1}$,
\begin{itemize}
\item[(6.2)]
$A([j-1] \cup S)$ and $B([2,j] \cup S)$
are directly connected on $\Lambda (S)$ by colors $H$,
\end{itemize}
where the above statement follows from (E1) when $j=1$.
Also, by (E3),
\begin{itemize}
\item[(6.3)]
if $i_1+1 \not \in S$ then
$A([j-1] \cup S)$ and $A([j-1] \cup\{i_1+1\} \cup S')$
are directly connected on $\Lambda(n,m)$ by color $i_1+1=j+2$.
\end{itemize}
Then it is straightforward that (i) and (ii) follow from (6.1), (6.2), (6.3)
and the definition of cancellations.
See Figure 5.

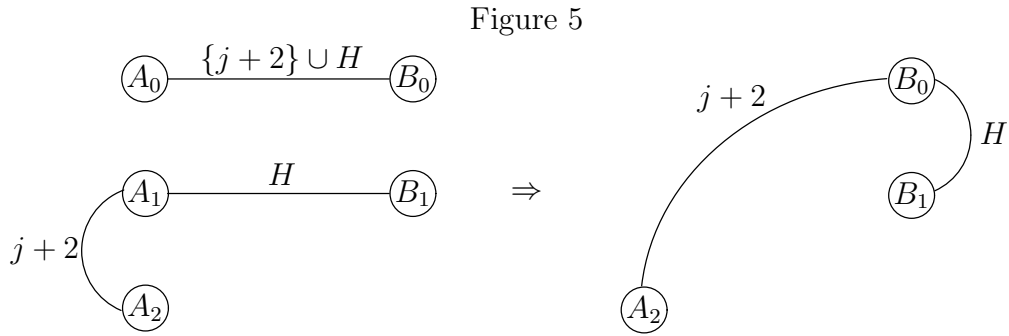
\begin{center}
%WinTpicVersion3.08
\unitlength 0.1in
\begin{picture}( 51.3000, 18.8000)(  7.0000,-23.0000)
% STR 2 0 3 0
% 3 1405 705 1405 805 5 0
% $A_0$
\put(14.0500,-8.0500){\makebox(0,0){$A_0$}}%
% CIRCLE 2 0 3 0
% 4 1405 805 1405 925 1405 925 1405 925
% 
\special{pn 8}%
\special{ar 1406 806 120 120  0.0000000 6.2831853}%
% STR 2 0 3 0
% 3 1410 1305 1410 1405 5 0
% $A_1$
\put(14.1000,-14.0500){\makebox(0,0){$A_1$}}%
% CIRCLE 2 0 3 0
% 4 1410 1405 1410 1525 1410 1525 1410 1525
% 
\special{pn 8}%
\special{ar 1410 1406 120 120  0.0000000 6.2831853}%
% STR 2 0 3 0
% 3 1410 1905 1410 2005 5 0
% $A_2$
\put(14.1000,-20.0500){\makebox(0,0){$A_2$}}%
% CIRCLE 2 0 3 0
% 4 1410 2005 1410 2125 1410 2125 1410 2125
% 
\special{pn 8}%
\special{ar 1410 2006 120 120  0.0000000 6.2831853}%
% STR 2 0 3 0
% 3 2810 705 2810 805 5 0
% $B_0$
\put(28.1000,-8.0500){\makebox(0,0){$B_0$}}%
% CIRCLE 2 0 3 0
% 4 2810 805 2810 925 2810 925 2810 925
% 
\special{pn 8}%
\special{ar 2810 806 120 120  0.0000000 6.2831853}%
% STR 2 0 3 0
% 3 2810 1305 2810 1405 5 0
% $B_1$
\put(28.1000,-14.0500){\makebox(0,0){$B_1$}}%
% CIRCLE 2 0 3 0
% 4 2810 1405 2810 1525 2810 1525 2810 1525
% 
\special{pn 8}%
\special{ar 2810 1406 120 120  0.0000000 6.2831853}%
% LINE 2 0 3 0
% 2 1525 805 2685 805
% 
\special{pn 8}%
\special{pa 1526 806}%
\special{pa 2686 806}%
\special{fp}%
% LINE 2 0 3 0
% 2 1525 1405 2685 1405
% 
\special{pn 8}%
\special{pa 1526 1406}%
\special{pa 2686 1406}%
\special{fp}%
% STR 2 0 3 0
% 3 2125 1205 2125 1305 5 0
% $H$
\put(21.2500,-13.0500){\makebox(0,0){$H$}}%
% STR 2 0 3 0
% 3 2125 605 2125 705 5 0
% $\{j+2\} \cup H$
\put(21.2500,-7.0500){\makebox(0,0){$\{j+2\} \cup H$}}%
% CIRCLE 2 0 3 0
% 4 1415 1705 1415 1365 1295 1385 1295 1985
% 
\special{pn 8}%
\special{ar 1416 1706 340 340  1.9756881 4.3536183}%
% STR 2 0 3 0
% 3 885 1605 885 1705 5 0
% $j+2$
\put(8.8500,-17.0500){\makebox(0,0){$j+2$}}%
% STR 2 0 3 0
% 3 3405 1305 3405 1405 5 0
% $\Rightarrow$
\put(34.0500,-14.0500){\makebox(0,0){$\Rightarrow$}}%
% STR 2 0 3 0
% 3 4020 1915 4020 2015 5 0
% $A_2$
\put(40.2000,-20.1500){\makebox(0,0){$A_2$}}%
% CIRCLE 2 0 3 0
% 4 4020 2015 4020 2135 4020 2135 4020 2135
% 
\special{pn 8}%
\special{ar 4020 2016 120 120  0.0000000 6.2831853}%
% STR 2 0 3 0
% 3 5420 715 5420 815 5 0
% $B_0$
\put(54.2000,-8.1500){\makebox(0,0){$B_0$}}%
% CIRCLE 2 0 3 0
% 4 5420 815 5420 935 5420 935 5420 935
% 
\special{pn 8}%
\special{ar 5420 816 120 120  0.0000000 6.2831853}%
% STR 2 0 3 0
% 3 5420 1315 5420 1415 5 0
% $B_1$
\put(54.2000,-14.1500){\makebox(0,0){$B_1$}}%
% CIRCLE 2 0 3 0
% 4 5420 1415 5420 1535 5420 1535 5420 1535
% 
\special{pn 8}%
\special{ar 5420 1416 120 120  0.0000000 6.2831853}%
% STR 2 0 3 0
% 3 5850 990 5850 1090 5 0
% $H$
\put(58.5000,-10.9000){\makebox(0,0){$H$}}%
% ELLIPSE 2 0 3 0
% 4 5410 2010 4000 800 5290 810 4010 1890
% 
\special{pn 8}%
\special{ar 5410 2010 1410 1210  3.2412613 4.6267620}%
% STR 2 0 3 0
% 3 3405 405 3405 505 5 0
% Figure 5
\put(34.0500,-5.0500){\makebox(0,0){Figure 5}}%
% STR 2 0 3 0
% 3 600 2200 600 2300 1 0
% $\left(\!\!\begin{array}{lll}A_0=A([j]\cup S'), A_1=A([j-1]\cup S), A_2=A([j-1]\cup\{i_1+1\} \cup S'),\\ B_0=B([j] \cup S'), B_1 = B([2,j+1] \cup S'), \D(S)= (A_0,A_1).\end{array}\!\!\right)$
\put(6.0000,-23.0000){\makebox(0,0)[lt]{$\left(\!\!\begin{array}{lll}A_0=A([j]\cup S'), A_1=A([j-1]\cup S), A_2=A([j-1]\cup\{i_1+1\} \cup S'),\\ B_0=B([j] \cup S'), B_1 = B([2,j] \cup S), \D(S)= (A_0,A_1).\end{array}\!\!\right)$}}%
% CIRCLE 2 0 3 0
% 4 5410 1100 5730 1100 5540 1390 5550 800
% 
\special{pn 8}%
\special{ar 5410 1100 320 320  5.1490161 6.2831853}%
\special{ar 5410 1100 320 320  0.0000000 1.1493771}%
% STR 2 0 3 0
% 3 4480 810 4480 910 5 0
% $j+2$
\put(44.8000,-9.1000){\makebox(0,0){$j+2$}}%
\end{picture}%

\end{center}
\bigskip
\bigskip

Note that Case 1 contains a proof of Lemma \ref{6.12} for $\D_1=\D_1([2,n+1])$
which is the starting point of the induction.
\medskip

\textit{Case 2.}
Suppose $j$ is odd and $i_1>j+1$.
Then
$$\D(S)=\big\{A\big([j-1] \cup S \big), B\big([i_1-j-1,i_1-2] \cup S'\big)\big\}.$$
By applying the induction hypothesis to $\hat S=\{i_1-1\} \cup S' \in X_j$,
\begin{itemize}
\item[(6.4)]
$B([i_1-j-1,i_1-2] \cup S')$ and $B([i_1-j,i_1-1] \cup S')$
are directly connected on $\Lambda(S)$ by colors $H$ (and $i_1+1$ if $i_1+1 \not \in S$).
\end{itemize}
By applying the induction hypothesis to $\tilde S= \{i_1-1 \} \cup S \in X_{j-1}$,
\begin{itemize}
\item[(6.5)]
$A([j-1] \cup S)$ and $B([i_1-j+1,i_1-1] \cup S)=B([i_1-j+1,i_1] \cup S')$
are directly connected on $\Lambda(S)$ by colors $H$,
\end{itemize}
where the above statement follows from (E1) when $j=1$.
Also, by (E3),
\begin{itemize}
\item[(6.6)]
if $i_1+1 \not \in S$ then
$A([j-1] \cup S)$ and $A([j-1] \cup\{i_1+1\} \cup S')$
are directly connected on $\Lambda(n,m)$ by color $i_1+1$.
\end{itemize}
Then it is straightforward that statement (i) and (ii) follow from (6.4), (6.5), (6.6)
and the definition of cancellations.
%See Figure 6.
%\medskip

%\begin{center}
%\input{Fig6}
%\end{center}
%\bigskip

\textit{Case 3.}
Suppose $j$ is even.
Then
$$\D(S)=\big\{ B \big([i_1-j,i_1-2] \cup S\big), B\big([i_1-j,i_1-1] \cup S'\big)\big\}.$$
We claim
\begin{itemize}
\item[(6.7)]
$A([j] \cup S')$ and $B([i_1-j,i_1-1] \cup S')$
are directly connected on $\Lambda(S)$ by colors $H'$.
\end{itemize}
If $i_1=j+1$ then (6.7) follows from (E1).
If $i_1>j_1+1$ then apply the induction hypothesis to $\hat S= \{i_1-1\} \cup S' \in X_j$.
Hence (6.7) holds.

Also,
by applying the induction hypothesis to $\tilde S= \{i_1-1 \} \cup S \in X_{j-1}$,
\begin{itemize}
\item[(6.8)]
$B([i_1-j,i_1-2] \cup S)$ and $B([i_1-j+1,i_1-1] \cup S)=B([i_1-j+1,i_1] \cup S')$
are directly connected on $\Lambda(S)$ by colors $H'$.
\end{itemize}
It is straightforward that statement (iii) follows from (6.7), (6.8)
and the definition of cancellations.
\end{proof}

\begin{proof}
[Proof of Lemma \ref{6.10}]
Let $S=\{i_1,\dots,i_{n+1-j}\}$.
We first prove that the vertices in $\D(S)$ are directly connected on $\Lambda(S)$ by color $i_1$.
If $j$ is even, then this is obvious since  by (E3) they are directly connected on $\Lambda(n,m)$ by color $i_1$.
Suppose $j$ is odd.
If $i_1=j+1$ then by (E3) they are directly connected on $\Lambda(n,m)$ by color $i_1$.
If $i_1>j+1$ then the claim follows by applying Lemma \ref{6.12}(ii) to $\{i_1-1\} \cup S'$.

It remains to prove that, for any $k \in \color (S) \setminus \{i_1\}$,
vertices in $\D(S)$ are directly connected on $\Lambda(S)$ by color $k$.
Let $k \in \color (S) \setminus \{i_1\}$ and $T=(S \setminus \{k\}) \cup \{k-1\}$.
Let $\D(T)=\{x,y\}$ and $\D(S)=\{z,w\}$.
Since $\D(T) \succ \D(S)$,
$x,y,z,w$ are vertices of $\Lambda(T)$.
By (E1), $x$ and $z$ are directly connected on $\Lambda(T)$ by color $k$.
Similarly $y$ and $w$ are directly connected on $\Lambda(T)$ by color $k$.
These facts say that $z$ and $w$
are directly connected on $\Lambda'(T)=\del_{\D(T)} \Lambda(T)$ by color $k$ (see Figure 6),
and therefore they are directly connected on $\Lambda(S)$ by color $k$.
\end{proof}

\begin{center}
%WinTpicVersion3.08
\unitlength 0.1in
\begin{picture}( 24.4000, 10.8500)(  8.6500,-13.0000)
% LINE 2 0 3 0
% 4 1000 600 1800 600 1800 1200 1000 1200
% 
\special{pn 8}%
\special{pa 1000 600}%
\special{pa 1800 600}%
\special{fp}%
\special{pa 1800 1200}%
\special{pa 1000 1200}%
\special{fp}%
% CIRCLE 2 0 2 0
% 4 1000 600 1000 700 1000 700 1000 700
% 
\special{pn 8}%
\special{sh 0}%
\special{ar 1000 600 100 100  0.0000000 6.2831853}%
% CIRCLE 2 0 2 0
% 4 1800 600 1800 700 1800 700 1800 700
% 
\special{pn 8}%
\special{sh 0}%
\special{ar 1800 600 100 100  0.0000000 6.2831853}%
% CIRCLE 2 0 2 0
% 4 1800 1200 1800 1300 1800 1300 1800 1300
% 
\special{pn 8}%
\special{sh 0}%
\special{ar 1800 1200 100 100  0.0000000 6.2831853}%
% CIRCLE 2 0 2 0
% 4 1000 1200 1000 1300 1000 1300 1000 1300
% 
\special{pn 8}%
\special{sh 0}%
\special{ar 1000 1200 100 100  0.0000000 6.2831853}%
% STR 2 0 3 0
% 3 1000 500 1000 600 5 0
% $x$
\put(10.0000,-6.0000){\makebox(0,0){$x$}}%
% STR 2 0 3 0
% 3 1800 500 1800 600 5 0
% $z$
\put(18.0000,-6.0000){\makebox(0,0){$z$}}%
% STR 2 0 3 0
% 3 1000 1100 1000 1200 5 0
% $y$
\put(10.0000,-12.0000){\makebox(0,0){$y$}}%
% STR 2 0 3 0
% 3 1800 1100 1800 1200 5 0
% $w$
\put(18.0000,-12.0000){\makebox(0,0){$w$}}%
% STR 2 0 3 0
% 3 1400 1010 1400 1110 5 0
% $k$
\put(14.0000,-11.1000){\makebox(0,0){$k$}}%
% STR 2 0 3 0
% 3 1400 410 1400 510 5 0
% $k$
\put(14.0000,-5.1000){\makebox(0,0){$k$}}%
% STR 2 0 3 0
% 3 2500 800 2500 900 5 0
% $\Rightarrow$
\put(25.0000,-9.0000){\makebox(0,0){$\Rightarrow$}}%
% STR 2 0 3 0
% 3 2400 200 2400 300 5 0
% Figure 7
\put(24.0000,-3.0000){\makebox(0,0){Figure 6}}%
% ELLIPSE 2 0 3 0
% 4 3200 900 3000 600 3200 600 3200 1200
% 
\special{pn 8}%
\special{ar 3200 900 200 300  1.5707963 4.7123890}%
% CIRCLE 2 0 2 0
% 4 3205 600 3205 700 3205 700 3205 700
% 
\special{pn 8}%
\special{sh 0}%
\special{ar 3206 600 100 100  0.0000000 6.2831853}%
% STR 2 0 3 0
% 3 3205 500 3205 600 5 0
% $z$
\put(32.0500,-6.0000){\makebox(0,0){$z$}}%
% CIRCLE 2 0 2 0
% 4 3205 1200 3205 1300 3205 1300 3205 1300
% 
\special{pn 8}%
\special{sh 0}%
\special{ar 3206 1200 100 100  0.0000000 6.2831853}%
% STR 2 0 3 0
% 3 3205 1100 3205 1200 5 0
% $w$
\put(32.0500,-12.0000){\makebox(0,0){$w$}}%
% STR 2 0 3 0
% 3 2900 800 2900 900 5 0
% $k$
\put(29.0000,-9.0000){\makebox(0,0){$k$}}%
\end{picture}%
\end{center}

\section{Real projective spaces and odd dimensional manifolds}

In this section,
we characterize face vectors of a few more classes of simplicial posets
by using Theorems \ref{1.2} and \ref{2.4}.

\subsection{Simplicial cell decompositions of $\RR P^n$}\

Face vectors of simplicial cell decompositions of the real projective space $\RR P^n$
were first studied by Masuda.
For integers $1 \leq i <n$,
let $r(n,i)= {n \choose i}$ if $i$ is even and $r(n,i)=0$ if $i$ is odd,
and let $r(n,n)=-1$ if $n$ is odd and $r(n,n)=0$ if $n$ is even.
The next result was proved by Masuda
except for the necessity of (3).

\begin{theorem}
\label{7.1}
A vector $h=(h_0,h_1,\dots,h_n) \in \ZZ^{n+1}$ is the $h$-vector
of a simplicial cell decomposition of $\RR P^{n-1}$ if and only if it satisfies
the following conditions:
\begin{itemize}
\item[(1)] $h_0=h_n -r(n,n)=1$ and $h_i-r(n,i)=h_{n-i} - r(n,n-i)$
for $i=1,2,\dots,n-1$.
\item[(2)] $h_i - r(n,i) \geq 0$ for $i =1,2,\dots,n-1$.
\item[(3)] if $h_i-r(n,i)=0$ for some $1 \leq i \leq n-1$ then
$h_0+h_1+ \cdots +h_n$ is even.
\end{itemize}
\end{theorem}

\begin{proof}
We work over a field $K$ of characteristic $2$.

(Necessity.)
Let $P$ be a simplicial cell decomposition of $\RR P^{n-1}$.
Then $P$ is an orientable homology manifold with $\beta_i(P)=1$ for all $i \geq 1$.
Thus $h_0''(P)=h_0,$
$h_i''(P)=h_i(P)-r(n,i)$ for $i=1,2,\dots,n$.
Then the desired conditions follow from Theorem \ref{2.4}.

(Sufficiency.)
By Corollary \ref{3.2},
what we must prove is the existence of a simplicial cell decomposition $P$
of $\RR P^{n-1}$ with $h''(P)=(1,0,\dots,0,1)$.
In the same way as in the proof of Lemma \ref{6.2},
it is enough to find a simplicial cell decomposition $P$ of $\RR P^{n-1}$
with
$$f_n(P)=2+ \sum_{i=1}^{n-2} \beta_i(P) { n-1 \choose i}= 2 ^{n-1}.$$
Consider the boundary complex $\partial \diamond^n$ of
the $n$-dimensional cross polytope $\diamond ^n \subset \RR^n$.
Thus $\diamond ^n$ is the convex hull of $\{\pm \mathbf e_i: i=1,2,\dots,n\}$,
where $\mathbf e_i$ is the $i$th unit vector of $\RR^n$.
Since cross polytope is simplicial and century symmetric
(say, if $F \subset \RR^n$ is a face of $\diamond ^n$ then $-F$ is also a face of $\diamond ^n$),
by identifying  $F$ and $-F$ for all faces $F$ of $\partial \diamond ^n$,
we obtain a simplicial cell decomposition of $\RR P^{n-1}$.
Since the number of the facets of $\diamond ^n$ is $2^n$,
the number of the facets of such a simplicial cell decomposition is $2^{n-1}$.
\end{proof}

\subsection{Odd dimensional manifolds}\

For a $(d-1)$-dimensional simplicial poset $P$,
the vector
$$\beta(P)=\big(1,\beta_1(P),\dots,\beta_{d-1}(P) \big) \in \ZZ_{\geq 0}^d$$
is called the \textit{Betti vector of $P$}.
If $P$ is an orientable homology manifold then the Poincar\'e duality guarantees the symmetry
$\beta_i(P)= \beta_{d-1-i}(P)$ for $i=1,2,\dots,d-2$.

For any vector $\beta=(1,\beta_1,\dots,\beta_{d-1}) \in \ZZ^d_{\geq 0}$
and a vector $h=(h_0,h_1,\dots,h_d) \in \ZZ^{d+1}$, we define
$h^\beta=(h_0^\beta,h_1^\beta,\dots,h^\beta_d)$ by
$h_0^\beta=h_0$, $h_k ^\beta= h_k - {d \choose k} \{ \sum_{\ell =2}^{k} (-1)^{\ell -k} \beta_{\ell-1}\}$ for $k=1,2,\dots,d-1$ and $h_d^\beta=h_d- \sum_{\ell =2}^{d-1} (-1)^{\ell -d} \beta_{\ell-1}$.
Thus, for a connected simplicial poset $P$,
if $h=h(P)$ and $\beta=\beta(P)$, then $h^\beta=h''(P)$.

\begin{theorem}
\label{7.2}
Let $d$ be an even number.
The vector $h=(h_0,h_1,\dots,h_d) \in \ZZ^{d+1}$
is the $h$-vector of a simplicial cell decomposition of a $(d-1)$-dimensional topological manifold without boundary if and only if there exists a symmetric vector
$\beta =(1,\beta_1,\dots,\beta_{d-1} ) \in \ZZ_{\geq 0}^d$ such that $h^\beta$
satisfies the conditions (1), (2) and (3) in Theorem \ref{1.1}.
\end{theorem}

\begin{proof}
By considering a field of characteristic $2$,
any topological manifold is an orientable homology manifold.
Then the necessity follows from Theorem \ref{2.4}.

We prove the sufficiency.
By Lemma \ref{3.1} and Corollary \ref{3.2},
for $(d-1)$-dimensional orientable manifolds $M_1$ and $M_2$,
if $\hh(M_1)=\hh(S^{d-1})$ and $\hh(M_2)=\hh(S^{d-1})$, then $\hh(M_1 \# M_2)=\hh(S^{d-1})$.
Since, for any symmetric vector $\beta=(1,\beta_1,\dots,\beta_{d-1}) \in \ZZ^d_{\geq 0}$,
we can make a $(d-1)$-dimensional manifold whose Betti vector is equal to $\beta$
from a sphere by taking a connected sum with the product of spheres repeatedly,
the statement follows from Theorem \ref{1.1}.
\end{proof}

\begin{remark}
The same argument characterizes all possible $h$-vectors of $(d-1)$-dimensional orientable
simplicial cell (homology) manifolds in characteristic $0$ when $d \not \equiv 3$ mod
$4$ since any Betti vector is attained by the same construction (see \cite{CTS}).

We also note that since $\hh(\RR P^2)= \hh(S^2)$ in characteristic $2$
by Theorem \ref{7.1},
and since $\hh(\CC P^2)=\hh(S^4)$ by the result of Gagliardi \cite{Ga} and Corollary \ref{3.2},
the same argument characterizes all possible face vectors of simplicial cell decompositions of $d$-dimensional topological manifolds without boundary for $d \leq 5$.
\end{remark}

As we suggested in section 3,
it would be interesting to find characterizations of face vectors of simplicial cell decompositions
of several types of manifolds.
For $3$-manifolds $M$,
it seems to be plausible that $\hh(M)$ has the unique minimal element.
(Indeed, this is true if we restrict the problem to graphical simplicial posets
since $h''$-vector decreases by  cancelling a dipole.)
Also, while we only consider manifolds without boundary in this paper,
it is of interest to consider manifolds with boundary.
The characterization of face vectors is open even for balls.
See \cite{Ko}.

%\noindent
%\textbf{Acknowledgments}:

\end{document}